\documentclass[alpha-refs]{wiley-article}

\usepackage{color}
\usepackage[normalem]{ulem}

\usepackage{graphicx,amsfonts,amssymb,amsmath,mathrsfs}

\usepackage{lineno}
\usepackage[breaklinks=true]{hyperref}
\usepackage{amsmath,amssymb}
\usepackage{caption} 
\usepackage{multirow}
\usepackage{xcolor}
\usepackage{breakcites}

\newtheorem{prop}{Proposition}
\newtheorem{thm}{Theorem}
\newtheorem{cor}{Corollary}
\newtheorem{ex}{Example}
\newtheorem{alg}{Algorithm}

\papertype{Original Article}

\title{Cox processes driven by transformed Gaussian processes on linear networks -- A  review and new contributions}

\author[1]{Jesper M\o ller}
\author[1]{Jakob G.\ Rasmussen}

\affil[1]{Department of Mathematical Sciences, Aalborg University, Skjernvej 4A, DK-9220~Aalborg~\O, Denmark}

\corraddress{Jesper M\o ller, Department of Mathematical Sciences, Aalborg University, Skjernvej 4A, DK-9220~Aalborg~\O, Denmark}
\corremail{jm@math.aau.dk}

\runningauthor{Jesper M\o ller and Jakob G.\ Rasmussen} 

\begin{document}

\maketitle

\begin{abstract}
{There is a lack of point process models on linear networks. For an arbitrary linear network, we consider new models for a Cox process with an isotropic pair correlation function obtained in various ways by
transforming an isotropic Gaussian process which is used for driving the random intensity function of the Cox process. In particular we introduce three model classes given by log Gaussian, interrupted, and permanental Cox processes on linear networks,
 and consider for the first time statistical procedures and applications for parametric families of such models.
 Moreover, we construct new simulation algorithms for Gaussian processes  on linear networks and
discuss whether the geode\-sic metric  or the resistance metric should be used for the kind of Cox processes studied in this paper.}

\textbf{Keywords} -- 
estimation,
geodesic distance,
interrupted point process,
isotropic covariance function,
log Gaussian Cox process,
moments, 
permanental Cox point process, 
resistance metric,
simulation.
\end{abstract}




%

\section{Introduction}\label{s:intro}

In Sir David R.\ Cox's highly influential paper `Some Statistical Methods Connected with Series of Events' \citep{Cox55} he invented  
doubly stochastic Poisson processes obtained by a generalization of Poisson processes where the intensity function $\Lambda$ that varies over space or time is a stochastic process. 
These Cox models play nowadays an important role when analysing point patterns in Euclidean spaces or on  spheres \citep[see][and the references therein]{MW04,spatstat,LawrenceEtAl}. 
In particular, Cox processes driven by a transformed Gaussian process (GP) $Y$ or independent copies $Y_1,\ldots,Y_h$ of $Y$
play a major role: A Log Gaussian Cox process (LGCP) has $\Lambda(u)=\exp(Y(u))$ \citep{LGCP,CM}; LGCPs constitute the most widely used subclass of Cox processes. Further, an interrupted Cox process (ICP) is obtained by an independent thinning of a Poisson process, where the selection probability of a point $u$ is given by  $\exp(-\sum_{i=1}^hY_i(u)^2)$ \citep[][]{Stoyan,FredMe}. Moreover, a permanental Cox point process (PCPP) is obtained if $\Lambda(u)=\sum_{i=1}^hY_i(u)^2$ \citep{Macchi:75,MacCM}.

In recent years there has been an increasing interest in analysing point patterns on a linear network $L$, that is, $L$ is a connected set in $\mathbb R^k$ (the real coordinate space of dimension $k$) given by a finite union of bounded, closed, line segments  
which can only overlap at their endpoints, see
\citet{AngEtAl}, \citet{spatstat}, the references therein as well as further references given later in the present paper.
Figure~\ref{fig:data} shows two examples of point patterns observed on linear networks, which we will use for illustrative purposes throughout the paper.
The first dataset was first analysed in \citet{AngEtAl} and consists of 116 locations of street crimes reported in the period 25 April to 8 May 2002 in an area close to the University of Chicago. 
The second dataset is 
one 
of the six point pattern datasets analysed in \citet{HeidiMe}  and 
consists of 566 locations of spines on a dendrite tree protruding from a neuron \citep[see also][]{RasmussenChristensen}.
 
\begin{figure}
	\begin{center}
		\includegraphics[width=7cm]{"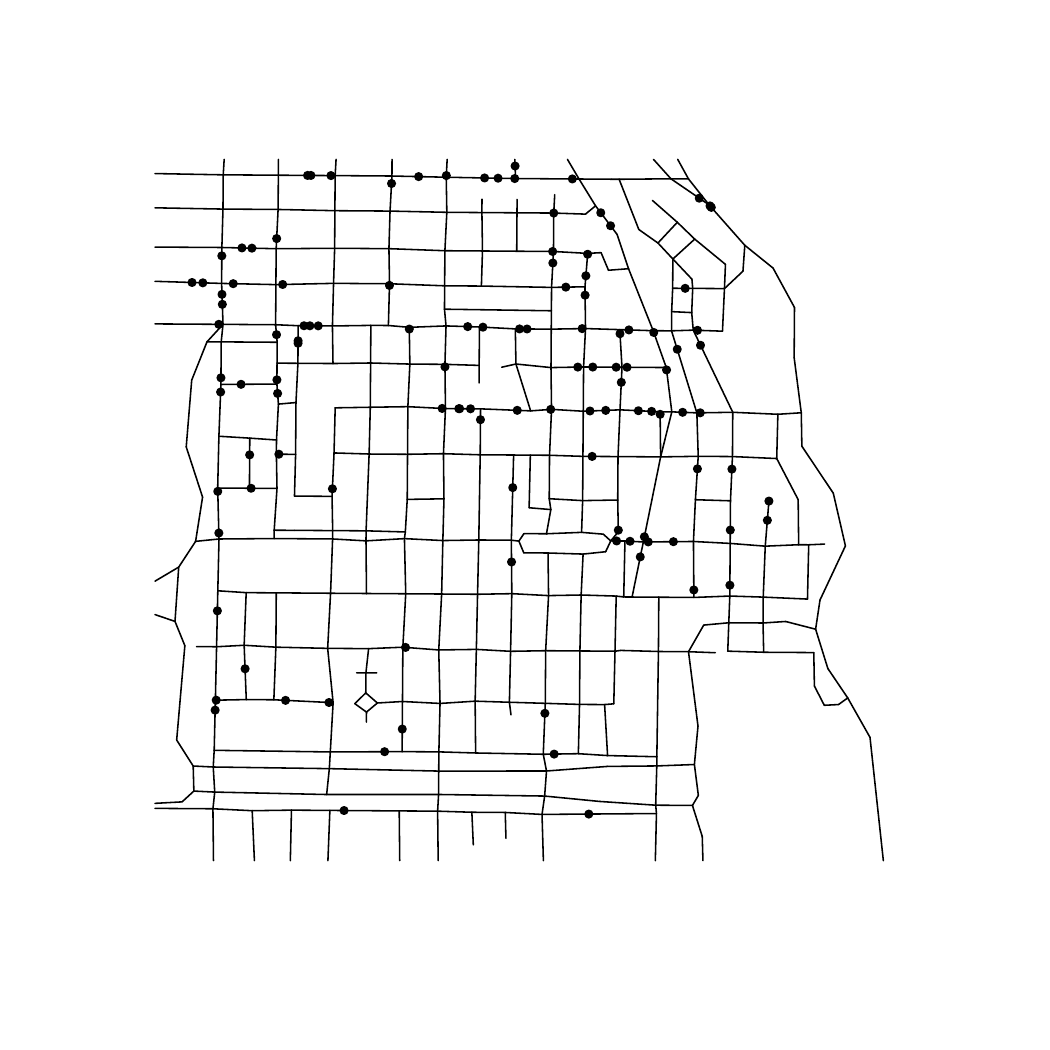"}
		\includegraphics[width=7cm]{"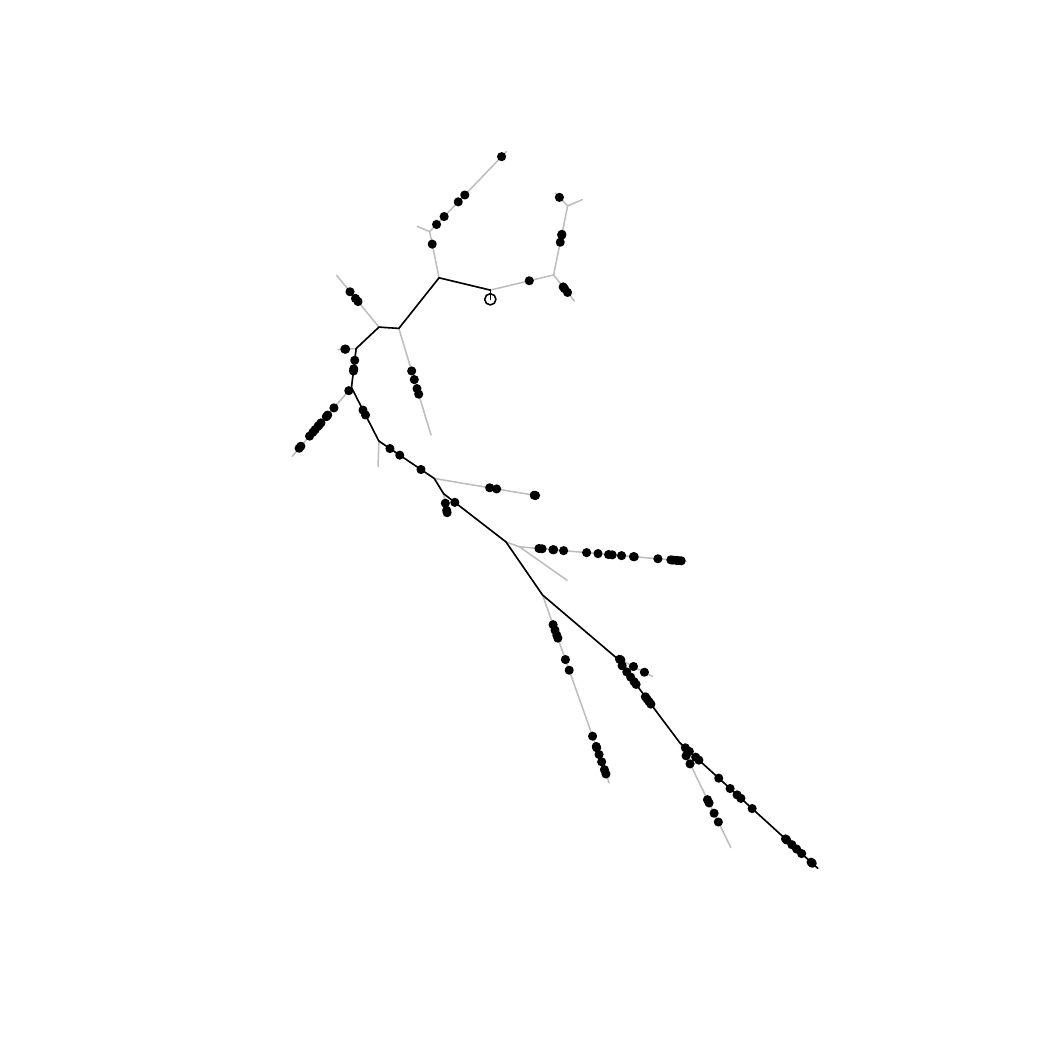"}
	\end{center}
	\caption{Left: The locations of street crimes in a part of Chicago. Right: The locations of spines on a dendrite tree \citep[dataset number five in][]{HeidiMe}. The circle marks the root of the tree, the black lines are a main branch, and the grey lines are side branches.}
	\label{fig:data}
\end{figure}

The contribution of the present paper is the following. We
 use isotropic covariance function models with respect to the geodesic metric $d_{\mathcal G}$ or the resistance metric $d_{\mathcal R}$
 on $L$ as
developed in \citet{AnderesEtAl} as well as new models developed in the present paper in order to construct models for isotropic GPs and hence new models for LGCPs, ICPs, and PCPPs on linear networks with isotropic pair correlation functions (more details follow in the next paragraph). Also we construct new simulation algorithms for GPs and consider for the first time statistical procedures and applications for parametric families of LGCPs, ICPs, and PCPPs on linear networks (however, our approach to parameter estimation only works well when using exponential covariance functions; see Sections~\ref{s:den-an}, \ref{s:simstudy}, and \ref{s:7.3} for details). Moreover, in continuation of the considerations in
\citet{AnderesEtAl} and \citet{RakshitEtAl}, comments in Sections~\ref{s:2}--\ref{s:conclusion} highlight the interest of $d_{\mathcal R}$ compared to $d_{\mathcal G}$.
Incidentally, we also establish new useful results for the resistance metric.

In brief, the paper consists of two parts, Sections~\ref{s:2}-\ref{s:GP-cov-fct} on our setting for isotropic covariance functions and related GPs, 
and Sections~\ref{s:3}-\ref{s:conclusion} on point processes, in particular Cox processes, including the cases of LGCPs, ICPs, and PCPPs on linear networks and how these models can be used for fitting real data. 
In more detail, the paper is organized as follows. 
Section~\ref{s:new-setting} discusses the definition of a linear network $L$ equipped with a metric $d$, where in particular we have in mind that $d$ is either $d_{\mathcal G}$ or $d_{\mathcal R}$. Section~\ref{s:res-met} gives a summary of results for $d_{\mathcal R}$, including a useful expression for the metric.  Section~\ref{s:iso-cov-fct} 
 studies isotropic covariance functions of the form $c(u,v)=c_0(d(u,v))$ and provides a less technical summary of results from \citet{AnderesEtAl} together with examples of isotropic covariance functions not appearing in that paper.
Simulation algorithms for GPs on $L$ with an isotropic covariance function are developed in Section~\ref{s:sim-GP}, where the case with $L$ a tree is particularly tractable. Our setting for point processes on linear networks is given in Section~\ref{s:3.1}, and Section~\ref{s:3.2} introduces first and higher order intensity functions which become useful when we later consider Cox process models. In particular, we focus on the pair correlation function and
the related $K$-function defined in Section~\ref{s:3.2.3}. As discussed in Section~\ref{s:est-check},  $g$, $K$, and other functional characteristics for point processes become useful for statistical inference. Section~\ref{s:cox-general} treats Cox processes on linear networks, and Section~\ref{s:models} surveys the properties of LGCPs, ICPs, and PCPPs models. Section~\ref{s:appl} demonstrates how these models may be fitted to real and simulated data. Finally, Section~\ref{s:conclusion} summarises our findings and discuss some open problems. 

At this point we should stress the importance of considering a pair correlation function (pcf) $g(u,v)$ with $u,v\in L$ to be isotropic, that is, $g(u,v)=g_0(d(u,v))$ for all $u,v\in L$:
\begin{description}
\item[(I)] it is easier to handle the one-dimensional function $g_0$ than the function $g$ defined on $L\times L$;
\item[(II)] as we shall see in Section~\ref{s:3.2.3}, the $K$-function is easily defined when $g$ is isotropic;
\item[(III)] to the best of our knowledge, 
 nonparametric estimators of  $g$ and $K$ have always been derived under the assumption that $g$ is isotropic;
\item[(IV)] moment based estimation procedures such as minimum contrast, composite likelihood, and Palm likelihood methods 
\citep[see][and the references therein]{MW07,MW17} 
become computational convenient if $g$ is isotropic;
 \item[(V)] in particular, for LGCPs, ICPs, and PCPPs, isotropy of $g$ becomes equivalent to isotropy of the covariance function for the underlying GPs;
 \item[(VI)] so far flexible model classes for covariance functions have mainly been developed in the isotropic case.
\end{description}
Indeed, assuming isotropy of the pcf has been a working assumption in most papers (including the present paper), but this assumption may of course be debated. For example, it means that the correlation is the same between two points independently of whether they are on the same line segment or not. In Section~\ref{s:7.1} we briefly discuss the interesting paper by \citet{BSW} which provides a new class of anisotropic Gaussian random fields.

A substantial part of this work was the development of an R-package \texttt{coxln}, in which the methods developed in the present paper are implemented. This package is available on Github under the author \texttt{gulddahl}. Moreover, we used the R-package \texttt{spatstat} extensively throughout the paper, see \cite{spatstat}. 

\section{Linear networks and metrics}\label{s:2}  

\subsection{Setting}\label{s:new-setting} 


This section specifies the setting of a linear network.

Denote $\mathbb R^k$ the $k$-dimensional Euclidean space, $k\in\{1,2,...\}$. For a linear network $L=\cup_{i=1}^m L_i$, we assume $m<\infty$, each $L_i\subset\mathbb R^k$ is a closed line segment of  length $l_i\in (0,\infty)$, $L_i\cap L_j$ is either empty or an endpoint of both $L_i$ and $L_j$ whenever $i\not=j$, and $L$ is a path-connected set. We equip $L$ with a `natural' metric $d$ (as given below) and arc length measure $\nu(A)=\int_A\mathrm d_L(u)$ for Borel sets $A\subseteq L$. We let $|L|=\nu(L)=\sum_{i=1}^m l_i$ denote the length of the linear network.

Various remarks are in order.

For disconnected linear networks, definitions and results may be applied
separately to each connected component of the network if we consider independent Gaussian processes on the connected components and independent point processes on the connected components.

The definition of a linear network may be extended to the more abstract case of a graph with Euclidean edges  \citep{AnderesEtAl} but 
we avoid this generalization 
for ease of presentation and 
since statistical methods have so far only been developed for the case (c). 

For each line segment $L_i$, there are two possible arc length parametrisations. We assume one is chosen and given by $u_i(t)=(1-t/l_i)a_i+(t/l_i)b_i$, $t\in[0,l_i]$, where $a_i$ and $b_i$ are the endpoints of $L_i$. The definitions and results in this paper will not depend on this choice, including when calculating arc length measure restricted to $L_i$: For Borel sets $A\subseteq L_i$, $\nu(A)=\int_0^{l_i}1(u_i(t)\in A)\,\mathrm dt$ where $1(\cdot)$ denotes the indicator function. Furthermore, let $V$ denote the set of endpoints of $L_1,\ldots,L_m$ and consider the graph with vertex set $V$ and edge set $E$ given by $L_1,\ldots,L_m$. Thus, two distinct vertices $u,v\in V$ form an edge if and only if $\{u,v\}=\{a_j,b_j\}$ for some $j\in\{1,\ldots,m\}$, in which case we write $u\sim v$. 

We have  
two cases of natural metrics in mind, namely when $d$ is the geodesic metric $d_\mathcal{G}$ or the resistance metric $d_\mathcal{R}$.
For $u,v\in L$, $d_\mathcal{G}(u,v)=\min\nu(p_{uv})$ where the minimum is over all paths $p_{uv}\subseteq L$ connecting $u$ and $v$.
Section~\ref{s:res-met} below provides the more technical definition of $d_\mathcal{R}$.

Indeed there are other interesting metric
including the least-cost metric \citep{RakshitEtAl}, 
but to the best of our knowledge parametric models for isotropic covariance functions $c(u,v)=c_0(d(u,v))$ 
have so far only been developed when $d=d_\mathcal{G}$, $d=d_\mathcal{R}$, or $d$ is given by the usual Euclidean distance. However, Euclidean distance is usually not a natural metric on a linear network.

\subsection{The resistance metric}\label{s:res-met} 

This section defines the resistance metric $d_\mathcal{R}$ for a graph with Euclidean edges \citep{AnderesEtAl} in the special case of
 a linear network $L=\cup_{i=1}^m L_i$ as given in case (c) in Section~\ref{s:new-setting}. The section also summarises some properties of $d_\mathcal{R}$ and compares with $d_\mathcal{G}$. 
 
 Consider the graph $G=(V,E)$ and its relation $\sim$ as defined above, and denote $d_V$  the classic (effective) resistance metric $d_V$ on $V$ \citep{KleinRandic}. 
 Since $d_\mathcal{R}$ is an extension of $d_V$ to $L$, we start by recalling the definition of $d_V$ using a notation as follows.
Let $u_0\in V$ be an arbitrarily chosen vertex called the origin.
For any $u,v\in V$, define the so-called conductance function by
\[{\mathrm{con}}(u,v)=\begin{cases*}
                    1/\|u-v\| & \text{if } $u\sim v$  \\
                     0 & \text{otherwise}
                 \end{cases*}
                 \] 
and define a matrix $\Delta$ with rows and columns indexed by $V$ so that its entry $(u,v)$ is given by
\[\Delta(u,v)=\begin{cases*}
                    1+c(u) & \text{if } $u=v=u_0$  \\
                     c(u) & \text{if } $u=v\not=u_0$\\
                     -{\mathrm{con}}(u,v) & \text{otherwise}
                 \end{cases*}
                 \] 
where $c(u)=\sum_{w\in V:\,w\sim u}{\mathrm{con}}(u,w)$ is the sum of the conductances associated to the edges incident to vertex $u$.
In fact $\Delta$ is symmetric and strictly positive definite, and it is similar to 
the `Laplacian matrix' obtained 
when viewing $G$ as an electrical network over the nodes with resistors given by the length of each line segment \citep[see e.g.][]{Kigami,JorgensenPearse} 
except that $\Delta$ has the additional $1$ added at entry $(u_0,u_0)$ (this makes $\Delta$ invertible). 

Now, let $B_0$ be a zero mean Gaussian vector indexed by $V$ and having covariance matrix $\Sigma=\Delta^{-1}$. Then
the resistance metric on $V$ is the variogram
\begin{equation}
\label{e:d-V}
d_V(u,v)=\mathbb{V}{\mathrm{ar}}(B_0(u)-B_0(v))=\Sigma(u,u)+\Sigma(v,v)-2\Sigma(u,v)\quad\mbox{for } u,v\in V.
\end{equation}
Extend $B_0$ by linear interpolation to a zero mean Gaussian process (GP) $Z_0$ on $L$ so that 
\[Z_0(u)=\frac{\|u-b_i\|}{l_i}B_0(a_i)+\frac{\|u-a_i\|}{l_i}B_0(b_i)\quad\mbox{for }u\in L_i.\]
For $i=1,\ldots,m$, define a zero mean Brownian bridge $B_i$ on $L_i$ so that  
\[\mathbb{C}{\mathrm{ov}}(B_i(u),B_i(v))=\min\{\|u-a_i\|\|v-b_i\|,\|v-a_i\|\|u-b_i\|\}/l_i\quad\mbox{for $u,v\in L_i$,}
\]
and define
\[Z_i(u)=\begin{cases*}
                    B_i(u) & \text{for } $u\in L_i$  \\
                     0 & \text{for } $u\in L\setminus L_i$.
                 \end{cases*}
                 \]
Finally, the resistance metric on $L$ is defined by 
 \begin{equation}\label{e:def-dR}
 d_{\mathcal{R}}(u,v) 
 =\sum_{i=0}^m \mathbb{V}{\mathrm{ar}}(Z_i(u)-Z_i(v))\quad\mbox{for } u,v\in L.
 \end{equation}
 
For the following theorem, which follows from \citet[][Propositions~2-4]{AnderesEtAl}, we use a terminology as follows. 
A closed line segment in $\mathbb R^k$ with endpoints $a$ and $b$ is denoted $[a,b]=\{at+b(1-t)\,|\, 0\le t\le 1\}$.
A path is a subset of $L$ of the form $[u,v_1]\cup [v_1,v_2]\cdots\cup [v_{i-1},v_{i}]\cup[v_{i},v]$ where $u,v\in L$, $v_1,\ldots,v_i\in V$ are vertices, $i\ge0$ is an integer, and we interpret $[v_1,v_2]\cdots\cup [v_{i-1},v_{i}]$ as the empty set if $i=0$. 
If all vertices in $G$ are of order two, we say that $L$ is a loop (since $L$ is isomorphic to a circle). If there is no loop, we say that $L$ is a tree.  

\begin{thm}\label{t:res-m} We have the following properties of $d_{\mathcal{G}}$ and $d_{\mathcal{R}}$.
\begin{description}
\item[(A)] The definition \eqref{e:def-dR} of $d_{\mathcal{R}}$ does not depend on the choice of origin $u_0\in V$.
\item[(B)] Both $d_{\mathcal{G}}$ and $d_{\mathcal{R}}$ are metrics on $L$, and their definitions are invariant to splitting a line segment $L_i$ into two line segments.
\item[(C)] For every $u,v\in L$, $d_\mathcal{G}(u,v)\ge d_\mathcal{R}(u,v)$, with equality if and only if there is only one path connecting $u$ and $v$. In particular,
$d_\mathcal{G}= d_\mathcal{R}$ if and only if $L$ is a tree.
\item[(D)] If $G$ is a loop, then $d_{\mathcal{R}}(u,v)=d_{\mathcal{G}}(u,v)-d_{\mathcal{G}}(u,v)^2/\sum_{i=1}^ml_i$.
\end{description} 
\end{thm}

While $d_{\mathcal{R}}$ is not as intuitive as $d_{\mathcal{G}}$,  it reflects the topology of $L$: Without loss of generality assume that $u,v\in V$, since if e.g.\ $u\in L_j\setminus V$, we may split $L_j$ into the two line segments with endpoints $\{a_j,v\}$ and $\{v,b_j\}$, and then consider a new graph with vertex set $V\cup\{u\}$ and edge set $E\cup\{\{a_j,v\},\{v,b_j\}\}$, cf.\ (B) in Theorem~\ref{t:res-m}. Viewing the graph $(V,E)$ as an electrical network with resistor $l_i$ at edge $L_i$, $i=1,\ldots,m$, we have that $d_\mathcal{R}(u,v)$ is the effective resistance between $u$ and $v$ as obtained by Kirkhoff's laws.  These laws are in accordance with (C). 
For example,  
for the Chicago street network in Figure~\ref{fig:data}, 
 $\max d_{\mathcal R}\approx 675$  
feet is much smaller than $\max d_{\mathcal G}\approx 2031$ feet, and it is also smaller than the side length of a square surrounding the network which is a little less than 1000 feet. 
Finally, the result in (D) quantifies for a circular network how much less $d_\mathcal{G}$ will be than $d_\mathcal{R}$. 
Of course it could be debated if using the resistance distance for the street network and in (D) are the right ways of quantifying connectedness, but at least we are not aware of 
any other metric on $L$ than $d_{\mathcal{R}}$ which reflects the topology and is useful for constructing valid isotropic pair correlation functions  (as considered in Section~\ref{s:GP-cov-fct} and later on). 
Moreover,
the following proposition and the remarks below show that $d_{\mathcal{R}}(u,v)$ is nicely behaving.
 
\begin{prop}\label{p:calc}  
For any $u\in L_j$ and $v\in L_i$, 
let 
\[s=\|u-a_j\|,\quad t=\|v-a_i\|,\quad A_i=d_V(a_i,b_i)/l_i^2-1/ l_i.\]
Then $A_i\le0$ with equality if and only if $L_i$ is the only path connecting $a_i$ and $b_i$, and $d_{\mathcal{R}}(u,v)$ satisfies the following. 
\begin{description}
\item[(A)] If $i=j$ then
\begin{equation}\label{e:d-R1}
d_{\mathcal{R}}(u,v) 
=\begin{cases}
A_i(t-s)^2+t-s & \text{if }t\ge s,\\
A_i(s-t)^2+s-t & \text{if }t\le s,
\end{cases}
\end{equation}
so $d_{\mathcal{R}}(u,v)$ considered as a function of $t$ is linear (the case $A_i=0$) or quadratic (the case $A_i<0$) on each of the intervals $[0,s]$ and $[s,l_i]$, continuous on $[0,l_i]$, and differentiable on $[0,l_i]\setminus\{s\}$.
\item[(B)] If $i\not =j$ then
\begin{equation}\label{e:d-R2}
d_{\mathcal{R}}(u,v) 
=A_it^2+B_{ij}(s)t+C_{ij}(s)
\end{equation}
where
\[B_{ij}(s)=1-\frac{2}{l_i}\left[\Sigma(a_i,a_i)-\Sigma(a_i,b_i)-\frac{l_j-s}{l_j}\Sigma(a_j,a_i)+\frac{l_j-s}{l_j}\Sigma(a_j,b_i)-\frac{s}{l_j}\Sigma(b_j,a_i)+\frac{s}{l_j}\Sigma(b_j,b_i)\right]\]
and
\[C_{ij}(s)=\frac{(l_j-s)^2}{l_j^2}\Sigma(a_j,a_j)+\frac{s^2}{l_j^2}\Sigma(b_j,b_j)+2\frac{s(l_j-s)}{l_j^2}\Sigma(a_j,b_j)+\Sigma(a_i,a_i)-2\frac{l_j-s}{l_j}\Sigma(a_j,a_i)-2\frac{s}{l_j}\Sigma(b_j,a_i)+\frac{s(l_j-s)}{l_j},\]
so $d_{\mathcal{R}}(u,v)$ is a linear or quadratic concave function of $t\in[0,l_i]$.
\item[(C)] If $i\not =j$ then
\begin{equation}
d_{\mathcal{R}}(u,v)\ge\min\{d_{\mathcal{R}}(u,a_i),d_{\mathcal{R}}(u,b_i)\}.\label{e:min}
\end{equation}
\end{description}
\end{prop}
\begin{proof} Since $l_i=d_{\mathcal G}(a_i,b_i)=d_V(a_i,b_i)$, Theorem~\ref{t:res-m}(C) gives that $A_i\le0$ with equality if and only if $L_i$ is the only path connecting $a_i$ and $b_i$. 
From \eqref{e:d-V} and \eqref{e:def-dR} we obtain \eqref{e:d-R1} and \eqref{e:d-R2} by a straightforward calculation, and thereby we easily see that $d_{\mathcal{R}}(u,v)$ as a function of $t$ behaves as stated in (A) and (B). 
Finally, since $d_{\mathcal{R}}(u,v)$ is a concave function of $t\in[0,l_i]$ in the case $i\not=j$, 
the inequality \eqref{e:min} follows.
\end{proof}


It follows from \eqref{e:d-R1} and \eqref{e:d-R2} that once $\Sigma=\Delta^{-1}$ has been calculated, $d_{\mathcal{R}}(u,v)$ can be quickly calculated for any $u,v\in L$. 
For example, 
the Chicago street network in Figure~\ref{fig:data} has 338 vertices, and using standard methods for inversion of the $338\times 338$ matrix $\Delta$ took less than a 0.1 second.
The inequality \eqref{e:min} becomes useful when searching for point pairs $u,v\in L$ with $d_{\mathcal{R}}(u,v)\le r$ and $v\in L_i$, since we need only to consider the cases where $d_{\mathcal{R}}(u,a_i)\le r$ or $d_{\mathcal{R}}(u,b_i)\le r$.

\section{Gaussian processes and isotropic covariance functions}\label{s:GP-cov-fct}

Let
$Y=\{Y(u)\,|\,u\in L\}$ be a Gaussian process (GP) 
where each $Y(u)$ is a real-valued random variable. The distribution of $Y$ is specified by the mean function $\mu(u)=\mathbb EY(u)$ and the covariance function 
\[c(u,v)=\mathbb C{\mathrm{ov}}(Y(u),Y(v))=\mathbb E[Y(u) 
{Y(v)}]-\mu(u) 
{\mu(v)}.\] 
The necessary and sufficient condition for a well-defined GP in terms of such two functions $\mu$ and $c$ is just that $c$ is symmetric  
and positive definite.

\subsection{Isotropic covariance functions}\label{s:iso-cov-fct}

We are in particular interested in isotropic covariance functions 
$c(u,v)=c_0(d(u,v))$ for all $u,v\in L$, where 
 with some abuse of terminology we also call $c_0$ a covariance function. So $c_0$ is required to be positive definite, that is, $\sum_{j,\ell=1}^n a_j{a_\ell}c_0(d(u_j,u_\ell))\ge0$ for all $a_1,\ldots,a_n\in\mathbb R$, all pairwise distinct $u_1,\ldots,u_n\in S$, and all $n=1,2,\ldots$. 
 
Henceforth, we assume that the variance $\sigma^2=c_0(0)$ is strictly positive, and
 pay attention to the correlation function $r_0(t)=c_0(t)/\sigma^2$.
 Many of the commonly used isotropic correlation functions, including those in Table~\ref{tab:covariances-sphere}, 
are valid with respect to the resistance metric but not always with respect to the geodesic metric. The reason for this is discussed 
in this section. 
In Table~\ref{tab:covariances-sphere}, for comparison we consider isotropic correlation functions defined on other metric spaces $(S,d)$, 
 where $d=\|\cdot\|$ is the usual Euclidean metric if $S=\mathbb R^k$, and where $d$ is the geodesic distance if 
 $S=\mathbb S^k$  is the $k$-dimensional unit sphere ($\mathbb S^k=\{x\in \mathbb R^{k+1}\,|\,\|x\|=1\}$).

\begin{table}
    \centering
    \begin{tabular}{l||l|l}
        Model & Correlation function $r_0(t)$ & Range of shape and smoothness parameters\\
        \hline
        \hline
        Powered exponential & $\exp\left(-t^\alpha/\phi\right)$ & $\alpha\in(0,2]$ if $S=\mathbb R^k$; $\alpha\in(0,1]$ if $S\in\{\mathbb S^k,L\}$\\
        Mat{\'e}rn & $\frac{2^{1-\alpha}}{\Gamma(\alpha)}\left( \sqrt{2\alpha}\frac{t}{\phi} \right)^{\alpha} K_{\alpha}\left(\sqrt{2\alpha}\frac{t}{\phi}\right)$ &  $\alpha>0$  if $S=\mathbb R^k$; $0< \alpha\leq \frac{1}{2}$ if $S\in\{\mathbb S^k,L\}$\\
        Generalized Cauchy & $(1 + (\frac{t}{\phi})^{\alpha})^{-\tau/\alpha}$ & $\tau > 0$; $\alpha \in (0,2]$ if $S=\mathbb R^k$; $\alpha \in (0,1]$ if $S\in\{\mathbb S^k,L\}$\\
        Dagum & $1 - ((\frac{t}{\phi})^{\tau}/(1+(\frac{t}{\phi})^{\tau}))^{\frac{\alpha}{\tau}}$ & $\tau \in (0,2]$ and $\alpha \in (0,\tau)$ if $S=\mathbb R^k$; \\
        & & $\tau \in (0,1]$ and $\alpha \in (0,1]$ if $S\in\{\mathbb S^k,L\}$
    \end{tabular}
    \caption{Four parametric models for an isotropic correlation function $r_0(r)$. Here, $\Gamma$ is the gamma function, $K_{\nu}$ is the modified Bessel function of the second kind, $\phi$ is a scale parameter, $\tau$ is a 
shape parameter, and $\alpha$ is a smooth\-ness parameter. The correlation functions are well-defined at all scales $\phi>0$ but the range of shape and smoothness parameters depend on the model and the space $S$. For $S\in\{\mathbb R^k,\mathbb S^k\}$, the correlation functions are well-defined for every dimension $k=1,2,\ldots$. For $S=L$, conditions on $L$ may be needed if distance is not measured by the resistance but the geodesic metric, see Section~\ref{s:iso-cov-fct}.
    } 
    \label{tab:covariances-sphere}
\end{table}



Typically (including all of our examples), $r_0$ will 
 be a completely monotone function. Recall that 
a function $f : [0,\infty)\mapsto\mathbb R$ is completely monotonic if it is non-negative and continuous on $[0,\infty)$ and for $j=0,1,\ldots$ and all $u>0$, the $j$-th derivative $f^{(j)}(u)$ exists and satisfies $(-1)^jf^{(j)}(u)\ge0$. By Bernstein's theorem, $f$ is completely monotone if and only if it is the Laplace transform of a non-negative finite measure on $[0,\infty)$, meaning that for every $t\ge0$,
\begin{equation}\label{e:bernstein}
f(t)=f(0)\int\exp(-st)\,\mathrm dF(s)
\end{equation}
where $F$ is a cumulative distribution function with $F(s)=0$ for $s<0$. We refer to $F$ as the Bernstein CDF corresponding to $f$. 

Thus, any non-negative valued distribution with a known Laplace transform can be used to produce a completely monotone function. This fact is used in the following example. 

\begin{ex}\label{ex:0} 
The following functions $f_1,f_2,f_3$ are completely monotone functions with $f_1(0)=f_2(0)=f_3(0)=1$ and they have corresponding Bernstein CDFs $F_1,F_2,F_3$ defined as follows.
For $\tau>0$, $\phi>0$, and $t\ge0$,
\begin{equation}\label{e:f_1}
f_1(t)=(1+t/\phi)^{-\tau},\quad F_1\sim\Gamma(\tau,\phi),
\end{equation}
where $\Gamma(\tau,\phi)$ denotes the gamma distribution with shape parameter $\tau$ and rate parameter $\phi$, and
\begin{equation}\label{e:f_2}
f_2(t)={2(t\phi)^{\tau/2}}K_\tau(2\sqrt{t\phi})/{\Gamma(\tau)},\quad F_2\sim\Gamma^{-1}(\tau,\phi),
\end{equation}
where $\Gamma^{-1}(\tau,\phi)$ denotes the inverse gamma distribution with shape parameter $\tau$ and scale parameter $\phi$.
Moreover, for $\psi>0$, $\chi>0$, $\lambda\in\mathbb R$, and $t\ge0$,
\begin{equation}\label{e:f_3}
f_3(t)= (1+2t/\psi)^{-\lambda/2}{K_\lambda(\sqrt{(2t+\psi)\chi})}/{K_\lambda(\sqrt{\psi\chi})}
\end{equation}
and $F_3$ is the CDF for a generalized inverse Gaussian distribution with probability density function
\[
\frac{(\psi/\chi)^{\lambda/2}} {2K_\lambda(\sqrt{\psi\chi})} s^{\lambda-1}\exp(-s \psi/2 -\chi/(2s)),\quad s\ge 0.
\]
\end{ex}

In the next theorem, which summarises Theorems 1 and 2 in \citet[][]{AnderesEtAl}, we need the following definition. We say that $L$ is a 1-sum of $\mathcal L_1=L_1\cup\ldots\cup L_j$ and $\mathcal L_2=L_{j+1}\cup\ldots\cup L_m$ if $\mathcal L_1$ and $\mathcal L_2$ are (connected) linear networks where $1\le j<m$, $\mathcal L_1\cap\mathcal L_2=\{u_0\}$ consists of a single point $u_0$, and 
\[d(u,v)=d(u,u_0)+d(v,u_0)\quad\mbox{whenever $u\in\mathcal L_1$ and $v\in\mathcal L_2$.}\]
This property is possible if $d=d_{\mathcal G}$ or $d=d_{\mathcal R}$ but unless $L$ is a straight line segment it is impossible if $d$ is given by the usual Euclidean distance. Using induction we say for $n=3,4,\ldots$ that $L=\mathcal L_1\cup\ldots\cup\mathcal L_n$ is a 1-sum of  $\mathcal L_1,\ldots,\mathcal L_n$ if $L$ is a 1-sum of $\mathcal L_1\cup\ldots\cup\mathcal L_{n-1}$ and $\mathcal L_n$. 


\begin{thm}\label{t:1}
Let $f : [0, \infty)\mapsto\mathbb R$ be a completely monotone and non-constant function. Then
 $f(d_{\mathcal R}(u, v))$ is strictly
positive definite over $(u, v)\in L\times L$. Moreover, if $L$ is a 1-sum of trees and loops, then $f(d_{\mathcal G}(u, v))$ is strictly
positive definite over $(u, v)\in L\times L$. However,
if there are three distinct paths between two points on $L$, then there exists a constant $\phi>0$ so that  $\exp(-d_{\mathcal G}(u, v)/\phi)$ is not positive definite over $(u, v)\in L\times L$. 
\end{thm} 

In Table~\ref{tab:covariances-sphere}, for
 each model, $r_0$ is completely monotone for the ranges of the parameters \citep[cf.\ the comments to Theorem~1 in][]{AnderesEtAl}. 
So for each example of $r_0$ in Table~\ref{tab:covariances-sphere} and for $r_0$ given by $f_1$, $f_2$, or $f_3$ in \eqref{e:f_1}--\eqref{e:f_3}, $r_0(d_{\mathcal R}(u, v))$ is a valid correlation function, but by Theorem~\ref{t:1} 
we only know that
$r_0(d_{\mathcal G}(u, v))$ is 
 valid if $L$ is a 1-sum of trees and loops. 
 In Table~\ref{tab:covariances-sphere}, the ranges of the parameters are the same for the two cases $S=\mathbb S^k$ and $S=L$ (in agreement with that $\mathbb S^1$ is isomorphic to $L$ if $L$ is  a loop). Finally, \eqref{e:f_1} is the special case of the generalized Cauchy function when $\alpha=1$ in Table~\ref{tab:covariances-sphere}, whilst \eqref{e:f_2} and \eqref{e:f_3} are not covered by Table~\ref{tab:covariances-sphere}. 
 
For example, consider the Chicago street network in the left panel of Figure~\ref{fig:data}. Here $L$ is not a 1-sum of trees and loops and therefore we cannot use the geodesic metric for the cases of covariance functions related to the Chicago street network. See also the counter examples in \citet[][Section~5]{AnderesEtAl}. On the other hand, the dendrite data shown in the right panel of Figure~\ref{fig:data} is observed on a tree, so here we can use the geodesic/resistance metric (by Theorem~\ref{t:1}(C) the two metrics are equal in this case).

\subsection{Simulation of GPs on linear networks}\label{s:sim-GP} 

This section discusses how to simulate a GP $Y=\{Y(u)\,|\,u\in L\}$ 
using three different algorithms, which are all available in our package \texttt{coxln}.
We assume without loss of generality that the mean function of $Y$ is zero.

The following is a straightforward algorithm applicable to any metric $d$ and any linear network $L$.

\begin{alg} Select a finite subset $D\subset L$ and make the following steps.
	\begin{itemize}\label{a:1}
		\item Simulate $Y$ restricted to $D$, e.g.\ by using  eigenvalue decomposition of the corresponding covariance matrix $\Sigma_D$.
		\item For $u\in L\setminus D$, approximate $Y(u)$ by the average of those $Y(v)$ where $v\in D$ is closest to $u$  with respect to $d_{\mathcal G}$.
	\end{itemize}
\end{alg}

Specifically, we have chosen a grid $D=V\cup D_1\cup…\cup D_m$ where each $D_i$ is a fine discretization of $L_i$ as described after the proof of Theorem~\ref{t:simalgotree} below.
The disadvantage of Algorithm~\ref{a:1} is of course that the dimension of $\Sigma_D$ can be large and hence eigenvalue decomposition (as well as other methods) can be slow. Algorithm~\ref{a:2} below is much faster but requires $L$ to be a tree and $c(u,v) = \sigma^2\exp(-s d(u, v))$ to be an exponential covariance function with parameter $s>0$ and $d=d_{\mathcal G}=d_{\mathcal R}$ (the exponential correlation function $r_0(t)=\exp(-ts)$ appears as two special cases in Table~\ref{tab:covariances-sphere} with scale parameter $\phi=1/s$, namely the powered exponential model with $\alpha=1$ and the M{\'a}tern model with $\alpha=\tfrac12$). But we first need to establish a Markov property given in the following theorem, where we denote the shortest path between $u,v\in L$ by $p_{uv}$.

\begin{thm}\label{t:markov}
	Suppose that $Y$ is a GP on a tree $L$ with exponential covariance function $c(u,v) = \sigma^2\exp(-s d(u, v))$ where $\sigma>0$, $s>0$, and $d=d_{\mathcal G}=d_{\mathcal R}$. 
	For $n=1,2,...$ and every pairwise distinct points $u,v,w_1,\ldots,w_n\in L$ so that $w_i\in p_{uv}$ for at least one $w_i$, we have that $Y(u)$ and $Y(v)$ are conditionally independent given $Y(w_1),\ldots,Y(w_n)$.
\end{thm}

\begin{proof}
	Let $n=1$ and $w=w_1$.
	Since $w\in p_{uv}$ and $L$ is a tree, $d(u,v) = d(u,w) + d(w,v)$ and therefore $c(u,v) = c(u,w)c(w,v)/\sigma^2$. Thus the covariance matrix for $(u,v,w)$ has the form 
	\begin{equation*}
	\Sigma_{u,v,w} = 
	\begin{pmatrix}
	\sigma^2 & c(u,w)c(w,v)/\sigma^2 & c(u,w) \\
	c(u,w)c(w,v)/\sigma^2 & \sigma^2 & c(w,v) \\
	c(u,w) & c(w,v) & \sigma^2
	\end{pmatrix}.
	\end{equation*}
	Inverting the covariance matrix, we get that the corresponding precision matrix has 0 at entries $(1,2)$ and $(2,1)$, thus implying that $Y(u)$ and $Y(v)$ are conditionally independent given $Y(w)$. 
	
	 Consider the case $n=2$ and e.g.\ $w_1\in p_{uv}$. Since $L$ is a tree, $Y(w_2)$ must be conditionally independent of either $Y(u)$ or $Y(v)$ given $Y(w_1)$. Assume without loss of generality that this is $Y(u)$. Thus, $Y(u)$ is conditionally independent of $(Y(v),Y(w_2))$ given $Y(w_1)$, which implies that $Y(u)$ and $Y(v)$ are conditionally independent given $(Y(w_1),Y(w_2))$. In a similar way we  verify the case with $n\ge3$. 
\end{proof}

We use the following notation in Algorithm~\ref{a:2}. Pick an arbitrary origin $u_0\in V$ and set $G_0(u_0)=\{u_0\}$. For $j=1,2,\ldots$, if $u\in V\setminus\cup_{i=0}^{j-1}G_i(u_0)$ and $u\sim v$ for some $v\in G_{j-1}(u_0)$, we call $u$ a child of $j$-th generation to $u_0$ and define $G_j(u_0)\subset V$ as the set of all children of $j$-th generation to $u_0$. 
Moreover, 
if $u\in G_{j-1}(u_0)$, $v\in G_j(u_0)$, and $u\sim v$, we define a GP $Y(u,v)=\{Y(w)\,|\,w\in(u,v]\}$ where $(u,v]$ is the half-open
line segment with endpoints $u$ and $v$ so that $u$ is excluded and $v$ is included. 

\begin{alg}\label{a:2}
Suppose that $L$ is a tree. Let $s>0$ and $\sigma>0$ be parameters, and pick an arbitrary vertex $u_0\in V$. Construct random variables $Y(w)$ for all $w\in L$ by the following steps.
	\begin{description}
		\item[(I)] For $w=u_0$, generate $Y(w)$ from $N(0,\sigma^2)$.
		\item[(II)] For $j=1,2,\ldots$, conditioned on all the $Y(w)$ so far generated,  generate independent GPs $Y(u,v)$ for all $u\in G_{j-1}(u_0)$ and all $v\in G_j(u_0)$ with $u\sim v$, where $Y(u,v)$ depends only on $Y(u)$ and for every $w,w_1,w_2\in(u,v]$ we have
		\begin{align}\label{e:M2}
		\mathbb E[Y(w)\,|\,Y(u)]&=\exp(-\|w-u\|s))Y(u)\\
		\label{e:M22}
		\mathbb C{\mathrm{ov}}[Y(w_1,w_2)\,|\,Y(u)]&=\sigma^2\left(\exp(-\|w_1-w_2\|s)-\exp(-\|w_1-u\|s-\|w_2-u\|s)\right).
		\end{align} 
		\item[(III)] Output $Y=\{Y(w)\,|\,w=u_0\mbox{ or }w\in (u,v]\mbox{ for some }j\in\mathbb N,\ u\in G_{j-1}(u_0),\ v\in G_j(u_0)\mbox{ with }u\sim v\}$.
	\end{description}
\end{alg}

\begin{thm}\label{t:simalgotree}
	Let the situation be as in Algorithm~\ref{a:2}.
	Then $Y$ is a zero mean GP on $L$ with exponential covariance function $c(u,v)=\sigma^2\exp(-s d(u, v))$. 
\end{thm} 

\begin{proof}
	By considering the GPs associated to the successive generations to $u_0$, we prove by induction that 
	the iterative construction in Algorithm~\ref{a:2} makes $Y$ a GP as stated. 
	
	Clearly the distribution of $Y(u_0)$ is correctly generated, cf.\ (I). This is induction step $j=0$. 
	
	For induction step $j\ge1$, we condition on all the $Y(w)$ so far generated, i.e., $w$ is equal to either $u_0$ or $Y(w)$ is a member of a GP associated to two neighbouring vertices, with one being a child of $i$-th generation $i$ to $u_0$ and the other being a child of $(i-1)$-th generation $i$ to $u_0$ where $1\le i\le j$, cf.\ (I) and (II) in Algorithm~\ref{a:2} 
	(here we interpret $u_0$ as a child of $0$-th generation to $u_0$). 
	By the induction hypothesis, these $Y(w)$ have been correctly generated. If we take two points $w_1$ and $w_2$ contained in different line segments between points in $G_{j-1}(u_0)$ and $G_j(u_0)$, then by Theorem~\ref{t:markov}, $Y(w_1)$ and $Y(w_2)$ are (conditionally) independent, which is in accordance to our construction. So it suffices to consider the (conditional) distribution of $(Y(w_1),Y(w_2))$ when $w_1,w_2\in(u,v]$, $u\in G_{j-1}(u_0)$, $v\in G_j(u_0)$, and $u\sim v$. This (conditional) distribution depends only on $Y(u)$, cf.\ Theorem~\ref{t:markov}, and it is straightforwardly seen to be  
a bivariate normal distribution with mean and covariance matrix given by \eqref{e:M2} and \eqref{e:M22}, respectively.	
\end{proof}

In practice, when using Algorithm~\ref{a:2} for simulation, a discretization on each line segment is needed. For some integer $n_j>0$ (which may depend on $l_j$) and $i\in\{0,1,\ldots,{n_j}\}$, define $s_i=i l_j/{n_j}$. Then, for
each $u\in L_j$, if $u$ is the midpoint between $s_{i-1}$ and $s_i$, we approximate $Y(u)$ by the average $(Y(u_j(s_{i-1}))+Y(u_j(s_i)))/2$, and otherwise if $u$ is closest to $u_j(s_i)$, we approximate $Y(u)$ by 
$Y(u_j(s))$. 
For example, if $u$ is closest to $u_j(s_i)$, the approximation error is bounded in probability by
\begin{equation}\label{e:bound-prob}
\mathrm P(|Y(u)-Y(u_j(s_i))|\ge t)=2\left(1-\Phi\left(\frac{t}{2\sigma^2(1-\exp(-s\|s_i-u\|))}\right)\right)\le 2\left(1-\Phi\left(\frac{t}{2\sigma^2(1-\exp(-sl_j/(2n_j)))}\right)\right)
\end{equation}
for $t\ge0$ where $\Phi$ is the distribution function for the standard normal distribution.
(Similarly, in case of Algorithm~\ref{a:1}, where in \eqref{e:bound-prob} we replace the exponential correlation function by the correlation function of $Y$.)
Further,
to generate the ${n_j}-1$ normal variables $Y(u_j(s_i))$, $i=1,\ldots,{n_j}-1$, we start by generating the variable $Y(u_j(s_1))$ in accordance to \eqref{e:M2} and \eqref{e:M22} with $w=w_1=w_2=s_1$. Then we can add $u_j(s_1)$ to the vertex set, whereby we split $l_j$ into the two line segments given by this new vertex and the endpoints of $l_j$. Hence, if $n_j>1$, we can repeat the procedure when generating $Y(u_j(s_2))$, and so on until all the $n_j-1$ normal variables have been generated.  

Theorems~\ref{t:markov} and \ref{t:simalgotree} do not hold if $L$ is not a tree. Indeed, a GP on the circle $\mathbb S^1$ (which is equivalent to a loop of length $2\pi$) with exponential covariance function, considering four arbitrary distinct points on $\mathbb S^1$, it can be shown that the GP is not  
Markov by calculating the precision matrix of the GP at these four points. On the other hand, letting $c(u,v) = a \cosh(b(d_\mathcal{G}(u,v) -\pi))$ for $a,b>0$,
\cite{Pitt71} verified that the GP on $\mathbb S^1$ with covariance function $c$ is Markov, but considering two arbitrary points on a tree together with a point on the path connecting them, it can be shown that the GP on the tree with covariance function $c$ is not Markov.
Consequently, we cannot have a covariance function only depending on the geodesic distance
which, for an arbitrary linear network, makes a GP Markov. Moreover, 
 in general, if we want to simulate a GP with an exponential covariance function on a linear network which is not a tree, we cannot rely on Markov properties. Instead we have to use the straightforward but slower Algorithm~\ref{a:1}.
 
For other covariance functions than the exponential, 
we may use the following theorem, which follows immediately from \eqref{e:bernstein} and the central limit theorem.

\begin{thm}\label{t:3}
Suppose $d$ is a metric on $L$ so that $(u,v)\mapsto\exp(-sd(u,v))$ for $(u,v)\in L^2$ is a well-defined correlation function for all $s>0$, and let $Y$ be a zero mean GP on $L$ with
	covariance function 
	\begin{equation}\label{e:c-ber}
	c(u,v)=\sigma^2
	\int\exp(-sd(u,v))\,\mathrm dF(s)
	\end{equation}
	where $\sigma>0$ and $F$ is a CDF with $F(s)=0$ for $s<0$. For an integer $n>0$ and $i=1,\ldots,n$, suppose we generate first $S_i$ from $F$ and second $Y_i$ as a zero mean GP on $L$ with covariance function $\sigma^2\exp(-S_i d(u,v))$ so that $(S_1,Y_1),\ldots,(S_n,Y_n)$ are independent. Let $\bar Y_n=\sum_{i=1}^n Y_i/n$. Then $\sqrt n \bar Y_n$ is a zero mean stochastic process on $L$ with covariance function $c$. As $n\rightarrow\infty$, $\sqrt n \bar Y_n$ approximates $Y$ in the sense that any finite dimensional distribution of $\sqrt n \bar Y_n$ converges in distribution towards the corresponding finite dimensional distribution of $Y$.  
\end{thm}   

Theorem~\ref{t:3} allows simulation of any GP with a covariance function of the form \eqref{e:c-ber},
if a simulation algorithm for $F$ is available. For the case of a tree, Theorem~\ref{t:3} combined with Algorithm~\ref{a:2} gives the following simulation algorithm.

\begin{alg}\label{a:3}
Let the situation be as in Theorem~\ref{t:3} and suppose that $L$ is a tree and $d=d_{\mathcal R}=d_{\mathcal G}$. 
Choose an integer $n>0$ and independently for $i=1,...,n$ make the following steps. 
\begin{description}
		\item[(I)] Generate $S_i$ from $F$.
		\item[(II)] Using Algorithm~\ref{a:2} generate $Y_i$ as a zero mean GP on $L$ with covariance function $\sigma^2\exp(-S_i d(u,v))$.
\end{description}
Output 	$\sqrt n \bar Y_n$ as an approximate simulation of 	a zero mean GP with covariance function given by \eqref{e:c-ber}.
\end{alg}

Conditioned on $S_1,...,S_n$, the output in Algorithm~\ref{a:3} is an approximate simulation of a zero mean GP with covariance function
\begin{equation}\label{e:approxc}
c(u,v|S_1,...,S_n)=\frac{\sigma^2}{n}\sum_{i=1}^n\exp(-S_i d(u,v)).
\end{equation}	
The running times of Algorithms~\ref{a:1}--\ref{a:3} for simulating an approximate GP on a tree are compared in the following example,  and the  choice of $n$ in Algorithm~\ref{a:3} needed to obtain that \eqref{e:approxc} is a good approximation of the covariance function \eqref{e:c-ber} is considered in the example thereafter.

\begin{ex} 
Let $L$ be the dendrite tree in Figure~\ref{fig:data}.
Table~\ref{tab:algotimes} shows the times used in our implementation for obtaining a simulation of a GP with either an exponential covariance function or a covariance function with inverse gamma Bernstein CDF, cf.\ \eqref{e:f_2}. 
Specifically, $\tau=\phi=\sigma=1$ but the running times do not depend on this choice of the parameters. However,
the running times depend heavily on the number of points $\#D$ in the grid $D\subset L$, so we consider two different situations, one with $\#D=387$ (this choice is used later in Sections~\ref{s:den-an} and \ref{s:simstudy}) and another with $\#D=1863$. Since Algorithm~\ref{a:3} depends on the choice of $n$, various values of $n$ are also shown in Table~\ref{tab:algotimes}. The first two rows in the table show that Algorithm~\ref{a:2} is faster than Algorithm~\ref{a:1} for simulating a GP with exponential covariance function for both choices of grids, but the difference is far more clear for the fine grid. For the inverse gamma Bernstein CDF in the case $\#D=387$, the running time of Algorithm~\ref{a:1} is roughly the same as for Algorithm~\ref{a:3} with $n=20$, while in the case $\#D=1863$ Algorithm~\ref{a:3} is much faster than Algorithm~\ref{a:1} even with $n=200$. This illustrates that Algorithm~\ref{a:2} is faster than Algorithm~\ref{a:1}, while it depends on the choice of $\#D$ and $n$ whether Algorithm~\ref{a:1} or \ref{a:3} is fastest.

\begin{table}
	\centering
	\begin{tabular}{l|c||c|c}
		Covariance function & Algorithm & $\#D = 387$ & $\#D = 1863$ \\
		\hline\hline
 		Exponential & Algorithm~\ref{a:1} & 0.129 s & 13.7 s \\
 		Exponential & Algorithm~\ref{a:2} & 0.0118 s &   0.0166 s \\
 		Inverse gamma Bernstein CDF & Algorithm~\ref{a:1} & 0.157 s & 14.5 s\\
 		Inverse gamma Bernstein CDF & Algorithm~\ref{a:3} (n=20) & 0.162 s & 0.255 s\\
 		Inverse gamma Bernstein CDF & Algorithm~\ref{a:3} (n=50) & 0.400 s & 0.638 s\\
 		Inverse gamma Bernstein CDF & Algorithm~\ref{a:3} (n=200) &  1.61 s & 2.48 s\\
	\end{tabular}
	\caption{Running times of Algorithms~\ref{a:1}--\ref{a:3} used for simulating a GP on two different grids on a tree.} 
	\label{tab:algotimes}
\end{table}
\end{ex}

\begin{ex} The choice of $n$ in Algorithm~\ref{a:3} leading to a good approximation depends on the choice of covariance function, so in 
continuation of the previous example, we consider two inverse gamma Bernstein CDFs with parameters 
$(\tau,\phi)=(1,1)$ and $(\tau,\phi)=(5,5)$. For these choices of parameters and for $n=20,50,200$, 
Figure~\ref{fig:algo3} shows the correlation functions along with 100 approximations given by \eqref{e:approxc}. For $n=20$ the approximations show a lot of variability around the true correlation functions, while for $n=200$ the approximations are much closer to the true correlation functions. Also it is evident in the figure that a higher $n$ is required to obtain a good approximation for $(\tau,\phi)=(1,1)$ than for $(\tau,\phi)=(5,5)$. This is expected since the variance of $\Gamma^{-1}(\tau,\phi)$ is infinite if and only if $0<\tau\le2$.

\begin{figure}
	\begin{center}
		\includegraphics[width=6cm]{"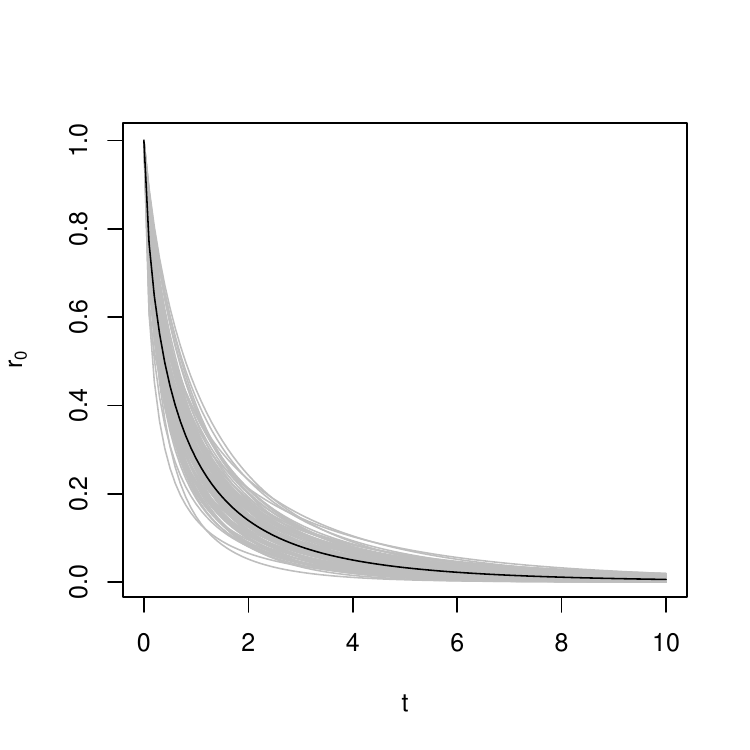"}
		\includegraphics[width=6cm]{"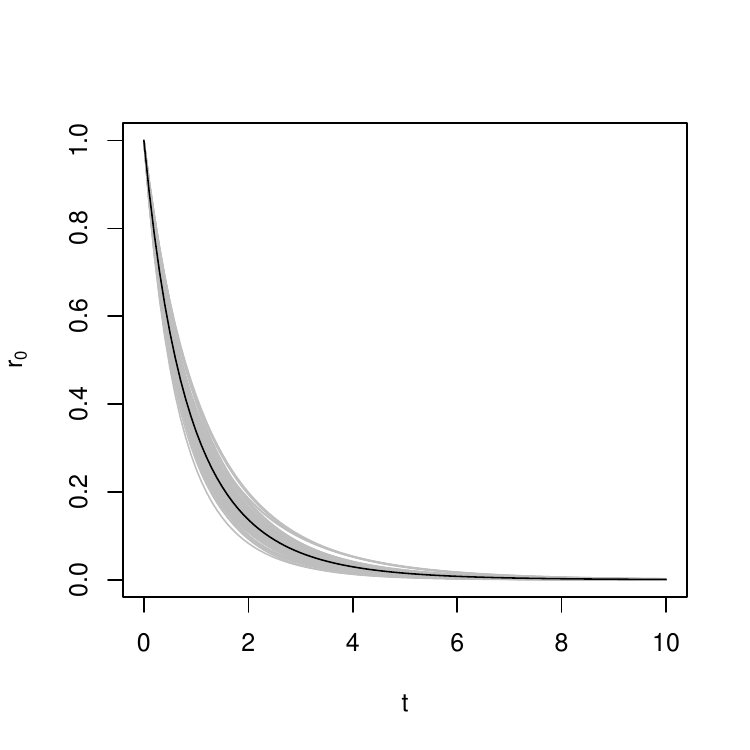"}\\
		\includegraphics[width=6cm]{"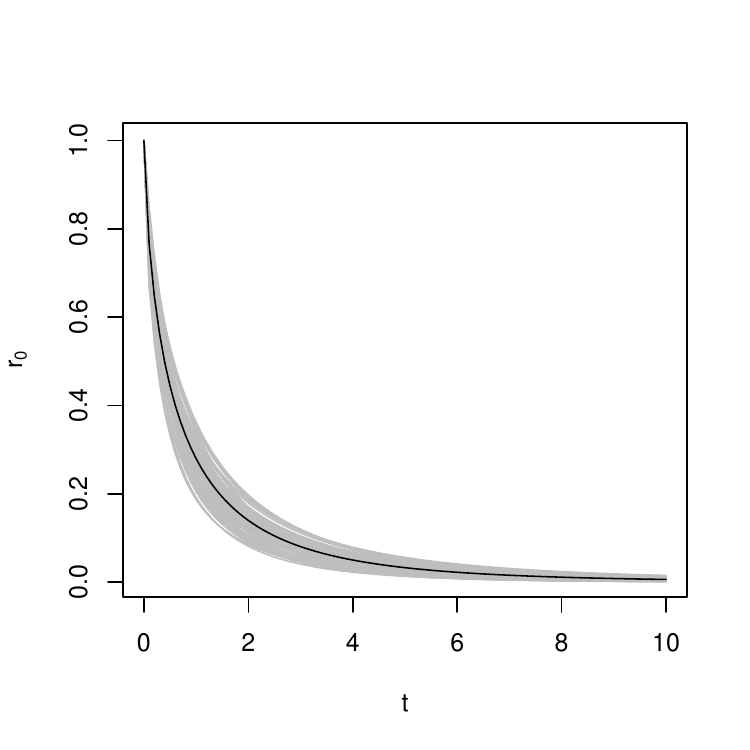"}
		\includegraphics[width=6cm]{"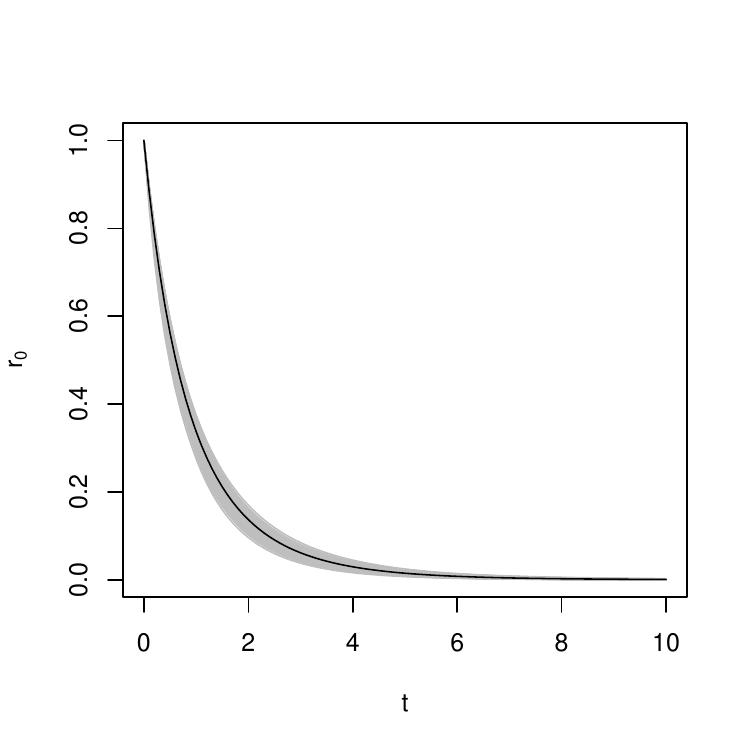"}\\
		\includegraphics[width=6cm]{"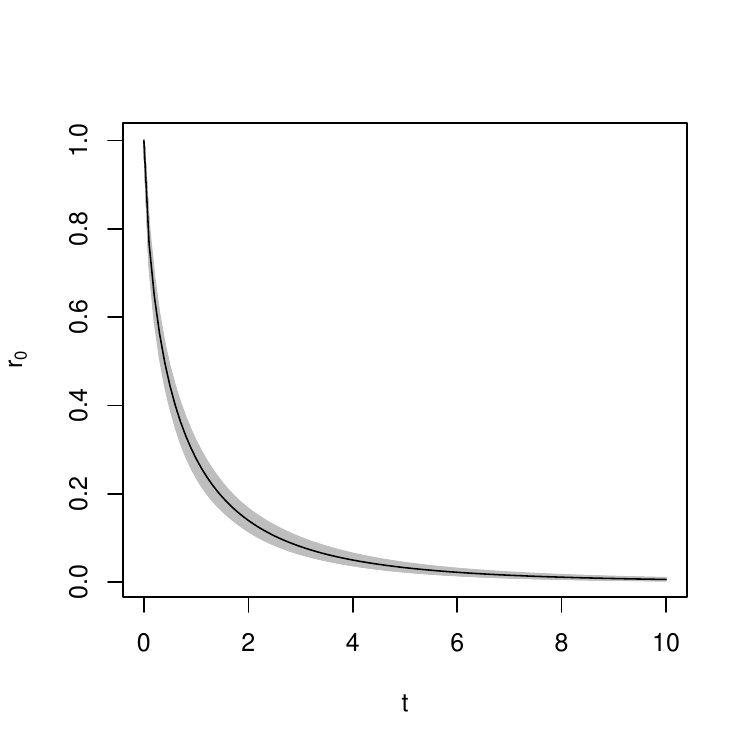"}
		\includegraphics[width=6cm]{"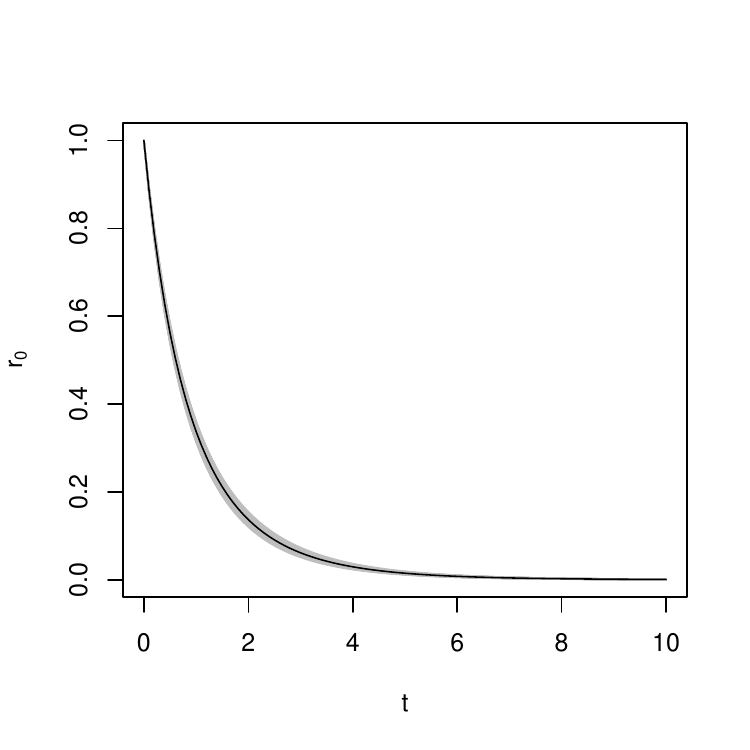"}
	\end{center}
	\caption{Correlations functions with inverse gamma Bernstein CDFs (black curves) with $(\tau,\phi)=(1,1)$ (left column) and $(\tau,\phi)=(5,5)$ (right column). The grey curves in each plot show 100 approximated correlation functions given by \eqref{e:approxc} with $n=20$ (upper row), $n=50$ (middle row), and $n=200$ (lower row).}
	\label{fig:algo3}
\end{figure}

\end{ex}

If $L$ is not a tree, Algorithms~\ref{a:2} and \ref{a:3} cannot be used, so we use the straightforward Algorithm~\ref{a:1} instead, provided of course that $Y$ is specified by a valid covariance function, meaning that Algorithm~\ref{a:1} may work for $d=d_{\mathcal{R}}$ but not for $d=d_{\mathcal{G}}$ as illustrated in the following example.

\begin{ex}  Figure~\ref{fig:chicago-GP} shows examples of simulations of zero mean GPs defined on the Chicago street network in Figure~\ref{fig:data} and with various isotropic covariance functions $c(u,v)=\sigma^2r_0(d(u,v))$ where $\sigma=1$ and in order that $c$ is valid we take $d=d_{{\mathcal R}}$, cf.\ Theorem~\ref{t:1}.  In the first two plots (the top row), $r_0(t)=\exp(-st)$ is 
an isotropic exponential correlation function with $s=0.1$ or $s=0.01$, and the next two plots (the middle row) relate to Theorem~\ref{t:3} with the Bernstein CDF given by a $\Gamma(1,100)$-distribution or a $\Gamma^{-1}(2,0.01)$-distribution, cf.\ Example~\ref{ex:0}. The densities for these Bernstein CDFs are shown in the left panel in the bottom row, and the last plot shows the   
corresponding correlation functions for $t\le200$ feet. 
The top row shows the scaling effect of the parameter $s$ for the exponential correlation function. Plots 2--4 are comparable, since $s$ has mean 0.01 in all three cases. 
Since $\max d_{\mathcal R}\approx 675$ feet, the last plot
indicates a rather strong correlation in the GPs in plots 2--4 and shows that the correlation is smallest when $r$ is fixed and rather similar when using the $\Gamma$ or $\Gamma^{-1}$ distribution although these distributions are rather distinct (e.g.\ the variance is finite for $\Gamma(1,100)$ but infinite for $\Gamma^{-1}(2,0.01)$). Accordingly, in plot 2 we see less smoothness than in plots 3 and 4 which show a similar degree of smoothness. 

\begin{figure}
	\begin{center}
	\includegraphics[width=6.5cm]{"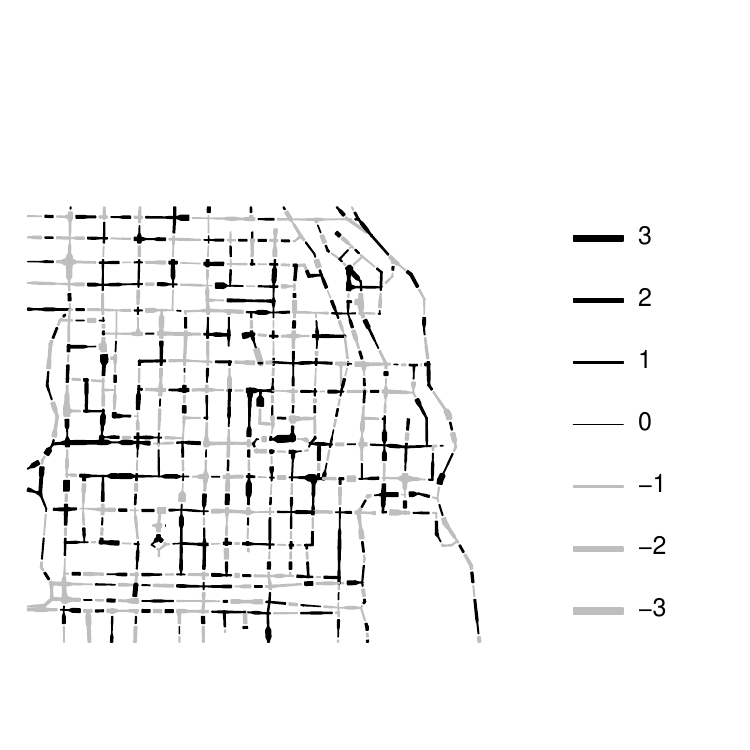"}
	\includegraphics[width=6.5cm]{"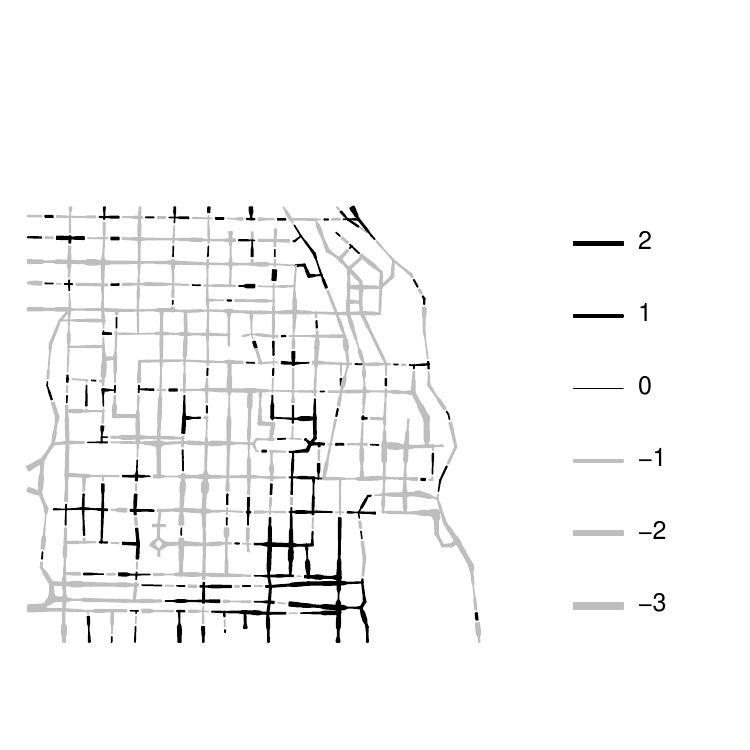"}\\
	\includegraphics[width=6.5cm]{"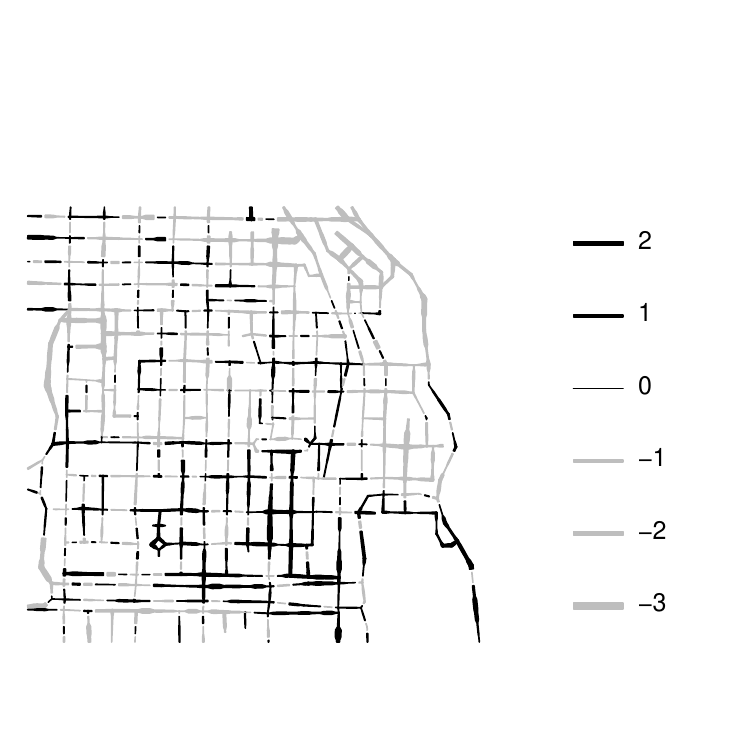"}
	\includegraphics[width=6.5cm]{"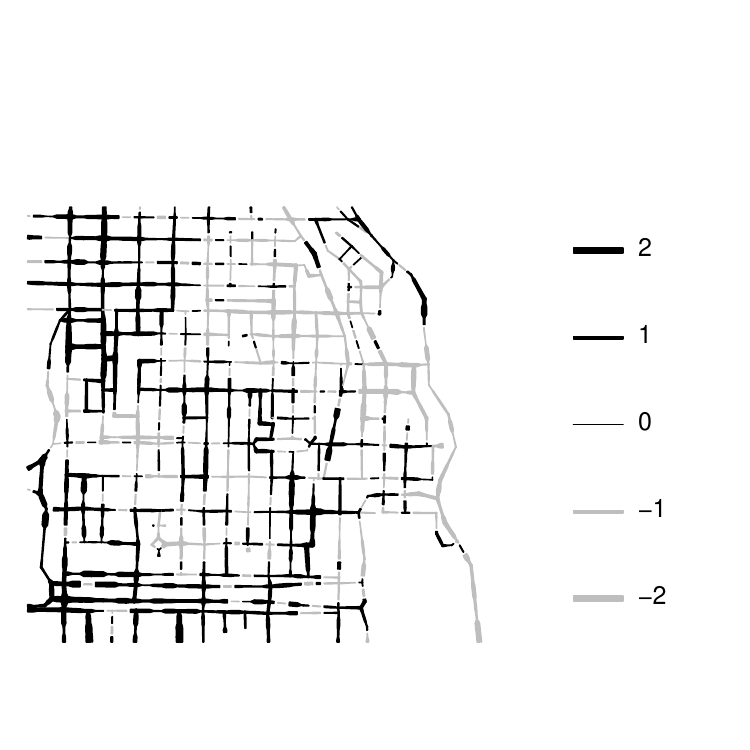"}\\
	\includegraphics[width=6cm]{"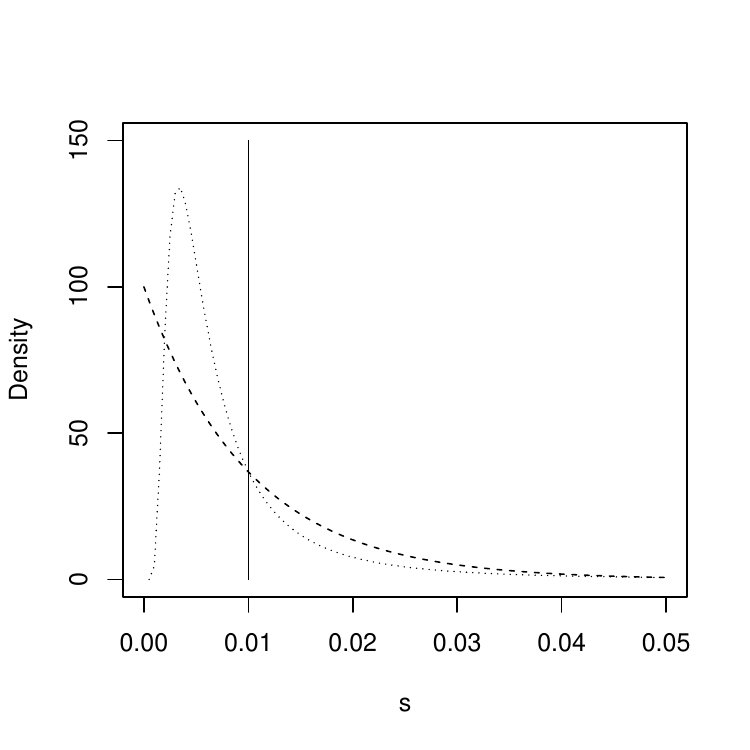"}
	\includegraphics[width=6cm]{"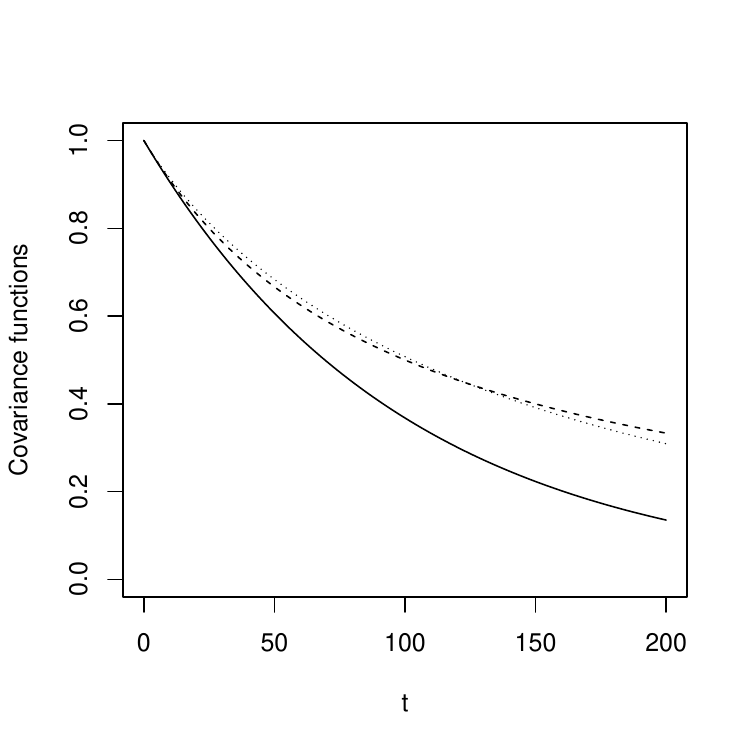"}
	\end{center}
	\caption{Simulation of zero mean GPs on the Chicago street network with an isotropic exponential covariance function $c(u,v)=r_0(d_{\mathcal{R}}(u,v))$. Top row: When $r_0(t)=\exp(-st)$ with parameter $s=0.1$ (left) or $s=0.01$ (right). Middle row: When $r_0$ has Bernstein CDF given by a gamma distribution (left) or inverse gamma distribution (right), where in both cases the mean of $s$ is 0.01. Bottom row: For plots 2--4, the corresponding densities for $F$ and correlation functions, where the curves in solid, dashed, and dotted correspond to plots 2--4, respectively.}
	\label{fig:chicago-GP}
\end{figure}

\end{ex}

\section{Point processes and some of their characteristics}\label{s:3}


This section reviews point processes, moment properties, and inference for point process models on linear networks, and the section provides the needed background material for Sections~\ref{s:cox}--\ref{s:appl}.
Readers who are familiar with the general theory for point processes defined on a metric space may glance many parts of Sections~\ref{s:3.1}, \ref{s:3.2}, and \ref{s:est-check}. Section~\ref{s:3.2.3} contains new results for the $K$-function on a linear network.

\subsection{Point process setting}\label{s:3.1}

 We restrict attention to point processes whose realisations can be viewed as finite subsets $x$ of $L$: Let $\mathcal B$ be the class of Borel sets $A\subseteq L$. We let the state space of 
 a point process on $L$ be $\mathcal N= \{x\subset L\,|\,\mbox{$x$ is finite} \}$. 
 For  
 $x\in \mathcal N$ and $A\in\mathcal B$, define $x_A= x\cap A$ and let $n(x_A)$ denote the cardinality of $x_A$.  
 Equip  $\mathcal N$ with the smallest $\sigma$-algebra $\mathcal F$ such that the mapping $x\mapsto n(x_A)$ is measurable for every $A\in\mathcal B$. Then
 by a point process is meant a random variable $X$ with values in $\mathcal N$ \citep[in the terminology of point process theory, $X$ is a simple locally finite point process, see e.g.][]{DaleyVere-Jones}. Equivalently, for every $A\in\mathcal B$, the count $N(A)=n(X_A)$ is a random variable.
 

We say that $X$ is a Poisson process with intensity function 
$\rho:L\mapsto[0,\infty)$ 
if 
$N(L)$ is Poisson distributed with finite mean $\int_L\rho(u)\,\mathrm d_L(u)$, and conditioned on $N(L)$, the points in $X$  are independent and each point has a density proportional to $\rho$ with respect to $\nu$. 

For every $u\in L$, we let $X_u$  be 
the point process which follows the reduced Palm distribution of $X$ at $u$, that is,
 \[\mathbb E\sum_{u\in X}h(X\setminus\{u\},u)=\int\rho(u)\mathbb E h(X_u,u)\,\mathrm d_L(u)\]
 for any non-negative measurable function $h$ defined on $\mathcal N\times L$ (equipped with the product $\sigma$-algebra of $\mathcal F$ and $\mathcal B$). Intuitively, $X_u$ follows the distribution of $X\setminus\{u\}$ conditioned on that $u\in X$ \citep[see e.g.][Appendix C]{MW04}. If $\rho(u)=0$, 
 $X_u$ may follow an arbitrary distribution. If   
  $X$ is a Poisson process with intensity function $\rho$, then $X$ and $X_u$ are identically distributed whenever 
 $\rho(u)>0$.
 
 Let $X_1$ denote the Poisson process with intensity 1. 
 Suppose 
 the distribution of $X$ is absolutely continuous with respect the distribution of $X_1$. 
 Let $f$ be the density of (the distribution of) $X$ with respect to (the distribution of) $X_1$.
 If $\rho(u)>0$, then $X_u$ has a density $f_u$ with respect to  $X_1$ such that
\begin{equation}\label{e:Palmdensity}
 f_u(\{u_1,\ldots,u_n\})=f(\{u,u_1,\ldots,u_n\})/\rho(u)\quad\mbox{for $n=1,2,\ldots$ and pairwise distinct }u_1,\ldots,u_n\in S\setminus\{u\}.
 \end{equation} 

\subsection{Moment properties}\label{s:3.2}
 
The point process $X$ has $n$-th order intensity function $\rho(u_1,\ldots,u_n)$ (with respect to the $n$-fold product measure of $\nu$) if this is a non-negative Borel function so that
\begin{equation}\label{e:1}
\mathbb E\left[N(A_1)\cdots N(A_n)\right]=\int_{A_1}\cdots\int_{A_n}\rho(u_1,\ldots,u_n)\,\mathrm d_L(u_1)\cdots\,\mathrm d_L(u_n)<\infty
\end{equation}
for every pairwise disjoint sets $A_1,\ldots,A_n\in\mathcal B$ ($\rho(u_1,\ldots,u_n)$ is also called the $n$-th order product density for the $n$-th order reduced moments measure). Thus, 
 $\rho(u_1,\ldots,u_n)$ is an integrable function, which is almost everywhere unique on $S^n$ (with respect to the $n$-fold product measure of $\nu$).
In the following, for simplicity nullsets are ignored, so the non-uniqueness of $\rho(u_1,\ldots,u_n)$ is ignored. Moreover, when we write $\rho(u_1,\ldots,u_n)$ it is implicitly assumed that the $n$-th order intensity function exists. 

In particular, $\rho(u)$ is the usual intensity function. The point process is said to be (first-order) homogeneous if $\rho(u)=\rho_0$ is constant. 

Instead of the second order intensity function, one usually considers the pair correlation function (pcf) given by
\[g(u,v)=\frac{\rho(u,v)}{\rho(u)\rho(v)},\]
 setting $\frac{a}{0}=0$ for $a\ge0$, and one usually assumes that
 \begin{equation}\label{e:iso-asump}
 g(u,v)=g_0(d(u,v))
 \end{equation}
 is isotropic. For a Poisson process, 
 \begin{equation}\label{e:Poisson-rho-n}
 \rho(u_1,\ldots,u_n)=\rho(u_1)\cdots\rho(u_n),
 \end{equation}
  so $g=1$. One often interprets $g_0(t)>1$ as repulsion/inhibition and $g_0(t)<1$ as attraction/clustering for point pairs at distance $t$ apart but one should be careful with this interpretation as $t$ grows. 

The specific models in this paper are typically attractive ($g>1$) and satisfies the following stronger property: For $n=2,3,\ldots$, $i=1,\ldots,n-1$, and any pairwise distinct $u_1,\ldots,u_n\in L$,
\begin{equation}\label{e:pos-cor}
\rho(u_1,\ldots,u_n)\ge \rho(u_1,\ldots,u_i)\rho(u_{i+1},\ldots,u_n),
\end{equation}
or equivalently, for any pairwise disjoint sets $A_1,A_2,\ldots\in\mathcal B$, 
\[\mathbb E\left[N(A_1)\cdots N(A_n)\right]\ge \mathbb E\left[N(A_1)\cdots N(A_i)\right]\mathbb E\left[N(A_{i+1})\cdots N(A_n)\right].\]
So \eqref{e:pos-cor} means that the counts $N(A_1),N(A_2),\ldots$ are positively correlated at all orders, and for brief we shall say that $X$ is positively correlated at all orders.  
 
Non-parametric estimation of $\rho$ and $g$ are discussed in 
\citet{RakshitEtAl}, \citet{RakshitEtAl19}, and \citet{BadEtAl2021}. 
For non-parametric estimation of the pcf, usually
kernel methods are used under the assumption \eqref{e:iso-asump} of isotropy (but we do not need to assume that $X$ is (first-order) homogeneous). 

\subsection{$K$-function}\label{s:3.2.3}

Since kernel methods are sensitive to the choice of bandwidth,  
popular alternatives  are given by 
estimators 
of the $K$-function 
defined in 
\eqref{e:K2new} 
below. On the other hand, for the specific parametric models in this paper, we have simple expressions for $g_0$ but not for $K$,
 and
 since the $K$-function is an accumulated version of $g_0$, it may be harder to interpret.

Suppose 
$g(u,v)=g_0(\delta(u,v))$ where $\delta$ is a metric on $L$ (since we consider derivatives below, it is convenient to switch from the previous notation $d$ to $\delta$). Then $X$ is said to be $\delta$-correlated, cf.\ \citet{RakshitEtAl}; if  $\delta=d_{\mathcal G}$, $X$ is also said to be second-order reweighted pseudostationary \citep{AngEtAl}.
Following \citet{RakshitEtAl} and defining $R=\inf_{u\in L}\sup_{v\in L}d(u,v)$, 
the $K$-function is given by
\begin{equation}\label{e:K2new}
K(t)= 
\int_0^t g_0(r)\,\mathrm dr,\quad 0\le t\le R.
\end{equation}
Note that $K$ depends only on $g_0$, but $g$ depends on both $g_0$ and $\delta$.
If $X$ is a Poisson process, then $K(t)=t$.

Non-parametric estimation of $K$ was carefully studied in \citet{RakshitEtAl} under the following technical assumption for the metric.  
Suppose $\delta$ is regular, meaning that for every $u\in L$, 
$\delta(u,v)$ is a continuous function of $v\in L$ and 
there is a finite set $N\subset L$ such that for $i=1,\ldots,m$ and all $v\in L_i\setminus N$, the Jacobian 
\begin{equation}\label{e:Jacobian}
J_\delta(u,v)=|(\mathrm d/\mathrm dt) \delta(u,v)|
\end{equation}
exists and is non-zero where $t=\|v-a_i\|$. 
 Both $d_\mathcal{G}$ and $d_\mathcal{R}$ are regular, where $J_{d_\mathcal{G}}=1$ and a useful expression for the calculation of $J_{d_\mathcal{R}}$ is given in the following corollary which follows immediately from Proposition~\ref{p:calc}.

\begin{cor}\label{c:weight}
For all $u\in L_j$ and $v\in L_i$ with $u\not=v$, 
using a notation as in Proposition~\ref{p:calc}, we have
 \begin{align}
\frac{\mathrm d}{\mathrm d t}d_{\mathcal{R}}(u,v) =
\begin{cases}
2A_i(t-s)+1 & \text{if }i=j,\ t> s,\\
2A_i(t-s)-1 & \text{if }i=j,\ t< s,\\
2A_it+B_{ij}(s) & \text{if }i\not=j.
\end{cases}
\label{e:der-d-R} 
\end{align}
\end{cor}

Some final remarks are in order. 

For $u\in L$ and $0\le t\le R$, define
\[w_\delta(u,t)=1\bigg/\sum_{v\in L:\, \delta(u,v)=t}1/J_\delta(u,v).\]
This is a weight which accounts for the geometry of the linear network when shifting from arc length measure on $L$ to Lebesgue measure on the positive half-line \citep[][Propositions 1 and 2]{RakshitEtAl}. 
For $\delta=d_{\mathcal{G}}$, we have 
$w_{\mathcal G}(u,t)=1/\#\{v\in L\,|\,\delta(u,v)=t\}$, since $J_{d_\mathcal{G}}=1$, and  
for $\delta=d_{\mathcal{R}}$, once the matrix $\Sigma$ from Section~\ref{s:res-met} has been calculated, $w_{d_{\mathcal{R}}}$ is quickly calculated from \eqref{e:der-d-R}.

It follows from \eqref{e:1}, \eqref{e:K2new}, and \citet[][Equation (8)]{RakshitEtAl}  that 
for any $A\in\mathcal B$ of positive arc length measure, 
\begin{equation}
K(t)=
\frac{1}{\nu(A)}\mathbb E\sum_{u\in X_A}\sum_{v\in X\setminus\{u\}}\frac{1(\delta(u,v)\le t)w_\delta(u,\delta(u,v))}{\rho(u)\rho(v)}.
\label{e:K3}
\end{equation}
Non-parametric estimators of $K$ are based on omitting the expectation symbol in \eqref{e:K3}, possibly after elaborating on the right hand side in \eqref{e:K3} in order to realize how correction factors can be included in order to adjust for edge effects, cf.\ \citet{RakshitEtAl}.  We may use \eqref{e:K3} as a more general definition of $K$ without assuming the existence of the pcf but requiring that the right hand side in \eqref{e:K3} is not depending on the choice of $A$, cf.\ \cite{BMW}.
 
In terms of Palm probabilities, for $d_L$-almost all $u\in L$ with $\rho(u)>0$,
\[K(t)=\frac{1}{|L|}\mathbb E \sum_{v\in X_u}\frac{1(\delta(u,v)\le t)w_\delta(u,\delta(u,v))}{\rho(v)}.\]
In general
the weight makes it hard to interpret this expression of $K$. 
 For point processes with points in $\mathbb R^k$ or $\mathbb S^k$, 
 there are much simpler expressions of $K$-functions in terms of Palm probabilities,
 see e.g.\ Baddeley et al.\ (2000) and M\o ller and Rubab (2016). 
 \nocite{BMW} \nocite{MollerRubak} 
 This is caused by that a translation is a natural transitive group action on $\mathbb R^k$ and a rotation is a natural transitive group action on $\mathbb S^k$. However, there is no natural transitive group action on a linear network.

\subsection{Estimation and model checking}\label{s:est-check} 
 
For parametric families of Cox point process models the most common estimation methods are based on the intensity, pair correlation, or $K$-functions using either minimum contrast estimation, composite likelihood, or Palm likelihoods, see \citet[][]{MW07,MW17} and the references therein. For the analyses in Section~\ref{s:appl}, we use minimum contrast estimation for fitting the Cox process models in Section~\ref{s:models} to various datasets in Section~\ref{s:appl} using the pcf (we discuss this choice of estimation method in Section~\ref{s:7.3}): If $g_0$ depends on a parameter, we estimate this parameter by minimizing the integral
\begin{equation}\label{e:contrast}
D(g_0,\hat g_0) = \int_{a_1}^{a_2} \left|g_0(t)^q-\hat{g}_0(t)^q\right|^p \mathrm{d}t
\end{equation}
where $\hat g_0$ is a non-parametric estimate of $g_0$ \citep[see e.g.\ (31) in][]{RakshitEtAl} and $0\leq a_1 < a_2$ and $p,q>0$ are user-specified values (see Section~\ref{s:appl}). The models in Section~\ref{s:models} all have nice expressions of the pcf but not of the $K$-function. Moreover, the models used for the data analyses in Section~\ref{s:appl} include a parameter for the intensity function, which $g_0$ does not depend on, and this parameter is estimated by a simple moment method; in the simplest case homogeneity is assumed and the intensity is estimated by $\hat{\rho}={n(x)}/{|L|}$. 

For model checking other functional summary statistics are needed when $\rho$, $g$, or $K$ and their corresponding non-parametric estimators have been used for estimation. 
\citet{CronieEtAl} suggested analogies to the so-called $F,G,J$-functions      \citep[introduced in][when considering point processes with points in $\mathbb R^k$]{Inhom-J-fct} which account for the geometry of the linear network. \citet[][]{CronieEtAl} 
showed that their $F,G,J$-functions make good sense under a certain condition called intensity reweighted moment pseudostationarity
(IRMPS): $X$ is IRMPS if $\inf\rho>0$ and $\delta$ is a regular metric on $L$ such that for $n=2,3,\ldots$, any pairwise distinct $u_1,\ldots,u_n\in L$, and any $u\in L$, $g(u_1,\ldots,u_n)=\rho^{(n)}(u_1,\ldots,u_n)/[\rho(u_1)\cdots\rho(u_n)]$ is of the form
\begin{equation}\label{e:IRMPS}
g(u_1,\ldots,u_n)=g_0(\delta(u,u_1),\ldots,\delta(u,u_n))
\end{equation}
for some function $g_0$. This condition is satisfied if  $X$ is either  a Poisson process or  
 a log Gaussian Cox process (LGCP) with a stationary pair correlation function (which is usually not a natural assumption, cf.\ Section~\ref{s:LGCP}). Apart from these examples 
\citet[][]{CronieEtAl} did not 
 verify any other cases of models where IRMPS is satisfied \citep[in Section~\ref{s:LGCP} we correct some mistakes in][]{CronieEtAl}. At least
\citet[][]{CronieEtAl} demonstrated 
the practical usefulness of their empirical estimator of the $J$-function for both a Poisson process, a simple
sequential inhibition (SSI) point process, and a LGCP. We show in Section~\ref{s:LGCP} that IRMPS is in general not satisfied for the LGCP. For the SSI point process in \citet[][]{CronieEtAl} it is hard to evaluate $g(u_1,\ldots,u_n)$ and hence to check the assumption of IRMPS. 

Alternatively, \citet{HeidiMe} introduced three purely empirical summary functions obtained by modifying the empirical $F,G,J$-functions for (inhomogeneous)
point patterns on a Euclidean space to linear
networks. Briefly, the modification consists of replacing the Euclidean space with
the linear network, introducing the shortest path distance instead of the Euclidean
distance, and adapting the notion of an eroded set to linear networks. \citet{HeidiMe} demonstrated the usefulness of their empirical $F,G,J$-functions for model checking although underlying theoretical functions are missing. We also use these functions for the data analyses in Section~\ref{s:appl}. Specifically, we will consider a concatenation of the three functions and validate a fitted model using a 95\% global envelope (i.e., a 95\%  confidence region for the concatenated function) and a $p$-value obtained by the 
global envelope test (based on the extreme rank length) as described in \citet{Mylly1}, \citet{Mylly2}, and \citet{Mylly3}. 
For the dendrite spine locations datasets analysed in \citet{HeidiMe} it is concluded that the results based on such a global envelope test are consistent with the results obtained if instead the summary functions from {\citet[][]{CronieEtAl}} are used.  

\section{Cox processes driven by transformed Gaussian processes}\label{s:cox}

This section introduces Cox processes on linear networks and in particular various kinds of model classes obtained by a transformed Gaussian process (interpreted as a random intensity function). Since such Cox process  models are introduced for the first time they deserve some attention, although readers who are familiar with similar models defined on $\mathbb R^k$ or $\mathbb S^k$ may glance many parts of Sections~\ref{s:cox-general}--\ref{s:models}. Moreover, Section~\ref{s:index} introduces an index of cluster strength.

\subsection{Cox processes on linear networks}\label{s:cox-general}

Let  $\Lambda=\{\Lambda(u)\,|\,u\in L\}$ be a non-negative stochastic process and for any  $A\in\mathcal B$,  define the random measure $\xi(A)=\int_A\Lambda(u)\,\mathrm d_L(u)$. Throughout this section we assume that almost surely $\xi(L)$ is finite, and we let
$X$ be a Cox process driven by $\Lambda$: That is, $X$ is a point process on $L$ which conditioned on $\Lambda$ is almost surely a Poisson process with intensity function $\Lambda$ (with respect to $\nu$). 
Usually in applications $\Lambda$ is unobserved, and so if only one point pattern dataset is observed the Cox process $X$ is indistinguishable from the inhomogeneous Poisson process $X|\Lambda$  \citep[for a discussion of which of the two models is most appropriate, see][Section 5.1]{MW04}. 

Henceforth, assume $\mathbb E\Lambda(u)$ is an integrable function with respect to $\nu$. Then the Cox process $X$ is well-defined, since almost surely $\xi(L)<\infty$.
By conditioning on $\Lambda$ it follows from \eqref{e:1} and \eqref{e:Poisson-rho-n} that
\begin{equation}\label{e:Cox-rho-n}
\rho(u_1,\ldots,u_n)=\mathbb E\left[\Lambda(u_1)\cdots\Lambda(u_n)\right].
\end{equation}

Let $W\subseteq L$ be a bounded Borel set with $\nu(W)>0$; we think of $W$ as an observation window. Then  $X_W$ is a Cox process driven by $\Lambda$ restricted to $W$, and $X_W$ has a density given by
\begin{equation}\label{e:densityCox}
f(\{u_1,\ldots,u_n\})=\mathbb E\left[\exp\left(\int_W (1- \Lambda(u))\,\mathrm d_L(u)\right)\prod_{i=1}^n\Lambda(u_i)\right]\quad\mbox{for pairwise distinct }u_1,\ldots,u_n\in W
\end{equation}
with respect to $X_1\cap W$, where $X_1$ is the unit rate Poisson process, cf.\ Section~\ref{s:3.1}.
In particular, if $W=L$ and $u\in L$ with $\rho(u)>0$, then $X_u$ has density
\begin{equation}\label{e:PalmdensityCox} 
f_u(\{u_1,\ldots,u_n\})=\mathbb E\left[\exp\left(\int_W (1- \Lambda(u))\,\mathrm d_L(u)\right)\frac{\Lambda(u)}{\rho(u)}\prod_{i=1}^n\Lambda(u_i)\right]
\end{equation}
with respect to $X_1$, cf.\ \eqref{e:Palmdensity}.
In general the densities in \eqref{e:densityCox} and \eqref{e:PalmdensityCox} are intractable because the expected values are difficult to evaluate. Instead we  
exploit \eqref{e:Cox-rho-n} for calculating the intensity and pair correlation functions and for making inference as demonstrated in the following sections. 

As pointed out in \citet{MW07} it is useful to write $\Lambda$ as
\begin{equation}\label{e:Lambda}
\Lambda(u)=\rho(u)\Lambda_0(u)
\end{equation}
where $\Lambda_0=\{\Lambda_0(u)\,|\,u\in S\}$ is a non-negative `residual' stochastic process with $\mathbb E\Lambda_0(u)=1$ whenever $\rho(u)>0$.
Then
\[\rho(u_1,\ldots,u_n)=\rho(u_1)\cdots\rho(u_n)\mathbb E\left[\Lambda_0(u_1)\cdots\Lambda_0(u_n)\right]\]
and 
 \begin{equation}\label{e:g-Lambda0}
 g(u,v)=\mathbb E\left[\Lambda_0(u)\Lambda_0(v)\right].
 \end{equation}
Typically, it is only $\rho(u)$ which is allowed to depend on covariate information, whilst $\Lambda_0$ is considered to account for unobserved covariates or other effects which has not been successfully fitted by a Poisson process with intensity function $\rho$, see e.g.\ \citet{MW07} and \citet{Diggle}. 

If 
\begin{equation}\label{e:all-orders}
\mathbb E\left[\Lambda_0(u_1)\cdots\Lambda_0(u_n)\right]\ge\mathbb E\left[\Lambda_0(u_1)\cdots\Lambda_0(u_i)\right]\mathbb E\left[\Lambda_0(u_{i+1})\cdots\Lambda_0(u_n)\right]
\end{equation}
for $n=2,3,\ldots$, $i=2,\ldots,n$, and all pairwise distinct $u_1,\ldots,u_n\in S$, then
$X$ is positively correlated at all orders, 
 cf.\ \eqref{e:pos-cor}. The condition \eqref{e:all-orders} will often be satisfied in the following.

\subsection{Models}\label{s:models}

Consider a GP $Y=\{Y(u)\,|\,u\in L\}$ with mean function $\mu$ and covariance function $c$, and let $Y_1,\ldots,Y_h$ be independent copies of $Y$.
In the remainder of this paper we study the following models.
\begin{itemize}
\item $X$ is a log Gaussian Cox process (LGCP) if 
\begin{equation}\label{e:defLGCP}
\Lambda_0(u)=\exp(Y(u))
\end{equation}
and $\mu(u)=-c(u,u)/2$ for all $u\in L$. The latter condition is required 
 since we want $\mathbb E\Lambda_0(u)=1$.  
\item $X$ is an interrupted Cox process (ICP) if 
 $\mu=0$ and 
\begin{equation}\label{e:defInter}
\Lambda_0(u)=\Pi(u)(1+2c(u,u))^{h/2}
\end{equation}
for all $u\in L$ where we define $\Pi(u)=\exp(-\sum_{i=1}^h Y_i(u)^2)$ (this definition differs slightly from the one used in \cite{FredMe}, which includes a factor $1/2$ inside the exponential function). Since $\mathbb E\Pi(u)=(1+2c(u,u))^{-h/2}$, we have $\mathbb E\Lambda_0(u)=1$. 
\item  $X$ is a permanental Cox point process (PCPP) if $\mu=0$, $c(u,u)=1$, and 
\begin{equation}\label{e:defPCPP}
\Lambda_0(u)=\frac{1}{h}\sum_{i=1}^h Y_i(u)^2
\end{equation}
for all $u\in L$. Since the sum in \eqref{e:defPCPP} is $\chi^2(h)$-distributed, $\mathbb E\Lambda_0(u)=1$.
\end{itemize}
In all cases, the distribution of $X$ is completely specified by $(\rho,c)$ and in the case of an ICP or PCPP the value of $h$. 
Note that the intensity function $\rho$ can be any non-negative locally integrable function with respect to $\nu$.

Similarly defined LGCP, ICP, and PCPP models are well-studied for point processes with points in $\mathbb R^k$ or $\mathbb S^k$, and most of their properties immediately extend to linear networks as summarized in the following.

\subsubsection{Properties of log Gaussian Cox processes}\label{s:LGCP}

Let $X$ be a LGCP, cf.\ \eqref{e:defLGCP}. 
As in \citet{LGCP} and \citet{CoeurjollyEtAl}, we obtain 
the following results. 
 For any integer $n\ge2$ and any pairwise distinct $u_1,\ldots,u_n\in L$,
\begin{equation}\label{e:n-LGCP}
\rho(u_1,\ldots,u_n)=\rho(u_1)\cdots\rho(u_n)\exp\bigg(\sum_{1\le i<j\le n}c(u_i,u_j)\bigg).
\end{equation}
In particular, the LGCP is determined by $\rho$ and 
\begin{equation}\label{e:pairLGCP}
g=\exp(c),
\end{equation} 
i.e., by its first and second order moment properties. Note that $g$ is isotropic if and only if $c$ is isotropic, and 
then $\rho(u_1,\ldots,u_n)$ depends only on the inter-point distances $d(u_i,u_j)$, $1\le i<j\le n$.  
Moreover, in most specific models, including those in Table~\ref{tab:covariances-sphere}, $c\ge0$ or equivalently $X$ is positively correlated at all orders, cf.\ \eqref{e:pos-cor} and \eqref{e:n-LGCP}. 
 
For $u\in L$ with $\rho(u)>0$, the reduced Palm process $X_u$ is a LGCP with intensity function $\rho(v|u)=\rho(v)\exp(c(u,v))$ but the pcf is still $g(v,w|u)=g(v,w)=\exp(c(v,w))$ for $v,w\in L$. This follows from \eqref{e:densityCox} and \eqref{e:PalmdensityCox} 
along similar lines as in \citet{CoeurjollyEtAl}.
Consequently, if $c$ is isotropic, the $K$-functions of $X$ and $X_u$ agree. 

Let us return to the concept of IRMPS as defined by \eqref{e:IRMPS}. \citet{CronieEtAl} noticed that IRMPS is satisfied  for the LGCP if $\inf\rho>0$ and for all $u_1,u_2,u\in L$, 
\begin{equation}\label{e:strong-cond}
c(u_1,u_2)=c_1(\delta(u,u_1),\delta(u,u_2))
\end{equation} 
for some function $c_1$ \citep[in our notation; see][Equation (29)]{CronieEtAl}. This statement is true due to \eqref{e:n-LGCP}, however, in our opinion \eqref{e:strong-cond} is a very strong condition, since we are not aware of any good examples where it is satisfied unless $L$ is isometric to a closed interval and $\delta$ is usual (Euclidean/geodesic/resistance) distance. 

Incidentally, in \citet[][Lemma 2]{CronieEtAl} the metric $\delta$ is assumed to be origin independent; they did not define the meaning of `origin independent' but we have been informed (by personal communication) that they had in mind that \eqref{e:strong-cond} should be  satisfied and 
when they let $\delta=d_{\mathcal R}$ be the resistance metric \citep[in the text after Lemma 2 in][]{CronieEtAl},
they admit that they misunderstood the meaning of origin independent as used in \citet[][Proposition 2]{AnderesEtAl}. Moreover, 
the proof of \citet[][Lemma 2]{CronieEtAl} is incorrect (they claim that $\delta(u',u_1)=\delta(u'',u_1)$ for any $u',u'',u_1\in L$, which is obviously not correct if $\delta=d_{\mathcal R}$, and which seems wrong in general for another choice of metric). 

\subsubsection{Properties of interrupted Cox processes}

Let $X$ be an ICP, cf.\ \eqref{e:defInter}. 
Then $X$ conditioned on $\Pi$ is obtained by an independent thinning of a Poisson process $Z$ on $L$ with intensity function $\rho_Z(u)=\rho(u)\left(1+c(u,u)\right)^{h/2}$, where the selection probabilities are given by $\Pi$. In the terminology of \citet{Stoyan}, $X$ is an interrupted point process. 


Assuming for ease of presentation  that $c(u,v)=\sigma^2r_0(d(u,v))$ is isotropic with $r_0(0)=1$, we obtain in a similar way as in \citet{FredMe} that the mean selection probability is constant and given by
\begin{equation}\label{e:pms}
p_{\mathrm{ms}}=(1+2\sigma^2)^{-h/2}
\end{equation}
and the pcf is given by
\begin{equation}\label{e:g0ICP}
g_0(t)=\left(\frac{(1+\sigma^2)^2}{(1+\sigma^2)^2-\sigma^4 r_0(t)^2}\right)^{h/2}.
\end{equation}
Third and higher-order moment results are less simple to express, but it can be proven that $X$ is positively correlated at all orders.
As $\sigma$ increases from 0 to infinity, then $p_{\mathrm{ms}}$ decreases from 1 to 0, whilst $g_0(t)$ increases from 1 to $(1-r_0(t)^2)-{h/2}$
if $r_0(t)\not=0$, which shows a trade-off between the degree of thinning and the degree of clustering. 
To understand how $g_0(t)$ depends on $h$ it
 is natural to fix the value of $p_{\mathrm{ms}}\in(0,1)$. 
   Then
  \[g_0(t)=\left(1+\left(1-p_{\mathrm{ms}}^{h/2}\right)^2r_0(t)^2\right)^{-h/2}\]
  is a strictly decreasing function of $h$ whenever $r_0(t)\not=0$.
Consequently, taking $h=1$ is natural if we wish to model as much clustering as possible when keeping the degree of thinning fixed. 

When $r_0(t)=\exp(-t/\phi)$ is an isotropic exponential covariance function, expressions of $K$ for $h=1,2,\ldots,5$ were given in \cite[][Appendix A]{HeidiMe}. Although these expressions were given for $d=d_{\mathcal G}$, they remain true for a general metric $d$ because the pcf depends only on $h$, $\sigma^2$, and $r_0(t)=\exp(-t/\phi)$, cf.\ \eqref{e:g0ICP} and our comment after \eqref{e:K2new}. 

\subsubsection{Properties of permanental Cox point processes}

Let $X$ be a PCPP, 
cf.\ \eqref{e:defPCPP}. Permanental Cox point processes with points in $\mathbb R^k$ were studied in \citet{Macchi:75} and \citet{MacCM}; using the parametrization in the latter paper, $X$ is a PCPP with parameters $\alpha=h/2$ and $C(u,v)=\sqrt{\rho(u)\rho(v)}c(u,v)/\alpha$. 

To stress that $c$ is a correlation function, we shall write $c=r$. For $n=1,2,\ldots$ and $u_1,\ldots,u_n\in L$, define the $\alpha$-weighted
permanent by
\[{\mathrm{per}}_\alpha[r](u_1,\ldots,u_n)=\sum_{\pi}\alpha^{\#\pi}r(u_1,u_{\pi_1})\cdots r(u_n,u_{\pi_n})
\]
where the sum is over all permutations $\pi=(\pi_1,\ldots,\pi_n)$ of $(1,\ldots,n)$ and $\#\pi$ is the number of cycles. 
The usual permanent corresponds to $\alpha = 1$ \citep{Minc}, in which case $X$ is also called a Boson process \citep{Macchi:75}. We have
\begin{equation*}\label{e:uha}
\rho(u_1,\ldots,u_n)=\rho(u_1)\cdots\rho(u_n)\mathrm{per}_\alpha[r](u_1,\ldots,u_n)/\alpha^n
\end{equation*}
from which it can be verified that 
 $X$ is positively correlated at all orders. It also follows that the degree of clustering is a decreasing function of $\alpha$. In particular,
\begin{equation}\label{e:pairPCPP}
g(u,v)=1+r(u,v)^2/\alpha.
\end{equation}
Thus $g\le 1+1/\alpha\le 3$, which reflects the limitation of modelling clustering by a PCPP.  

\citet{Valiant} showed that exact computation of permanents of general
matrices is a $\#$P (sharp P) complete problem, so no deterministic polynomial time algorithm is available. 
For most statistical purposes, approximate computation of
permanent ratios is sufficient, and analytic approximations are available for large $\alpha$. However, as $\alpha\rightarrow\infty$, $X$ tends to a Poisson process and the process becomes less and less interesting for the purpose of modelling clustering. 

The PCPP can be extended to the case where $\alpha\ge0$ and $c$ is not necessarily a covariance function (in which case we loose the connection to Gaussian and Cox processes), see \citet{MacCM} and \citet{ShiraiTakahashi}. Indeed the process also extends to the case where $\alpha$ is a negative integer whereby a (weighted) determinantal point process (DPP) is obtained. We return to DPPs in Section~\ref{s:7.1}. 

\subsection{An index of cluster strength}\label{s:index}

Consider again a LGCP, ICP, or PCPP $X$  with an isotropic covariance function $c(u,v)=\sigma^2r_0(d(u,v))$ where $r_0(0)=1$, $h=1$ if $X$ is an ICP, and $\sigma=1$ if $X$ is a PCPP. 
To quantify how far $X$ is 
from a Poisson process we follow 
\citet{BadEtAl2022} in defining 
\begin{equation}\label{e:varphi}
\varphi=g_0(0)-1.
\end{equation}
If $X$ has constant intensity $\rho$,
the total variation distance between the distribution of $X$ and that of a Poisson process with intensity $\rho$ is at most $\rho\nu(S)\sqrt\varphi$. This follows since 
\citet[][Lemma~9]{BadEtAl2022} immediately extends to our  setting of linear networks.

\citet{BadEtAl2022} used $\varphi$ to describe the degree of clustering for another class of Cox processes, namely Neyman-Scott point processes (on $\mathbb R^k$ and with the cluster size following a Poisson distribution): the degree of clustering increases as $\varphi$ increases. 
Although our Cox processes are not cluster point processes, $\varphi$ has a similar interpretation since realizations of the processes may look more or less clustered. More precisely, 
it follows from \eqref{e:pairLGCP}, \eqref{e:g0ICP}, and \eqref{e:pairPCPP} that $\varphi$ is equal to
\begin{equation}\label{e:varphiICP}
\exp(\sigma^2)-1\mbox{ if $X$ is a LGCP},\quad 
\frac{1+\sigma^2}{\sqrt{1+2\sigma^2}}-1\mbox{ if $X$ is an ICP},\quad 
2/h \mbox{ if $X$ is a PCPP,}
\end{equation}
and $\varphi+1$ is the maximal value of the pcf.
Thus, if $X$ is a LGCP or an ICP, $\varphi$ is a strictly increasing function of $\sigma>0$, where $X$ approaches a Poisson process as $\sigma\rightarrow0$, while realizations of $X$ become more and more clustered as $\sigma$ grows. If instead $X$ is a PCPP, the degree of clustering decreases as $h\in\{1,2,...\}$ increases, where $X$ approaches a Poisson process as $h\rightarrow\infty$. 

The index $\varphi$ for `cluster strength' is used in Sections~\ref{s:simstudy} and \ref{s:7.3}.

\section{Application of statistical inference procedures}\label{s:appl}

For the analyses of the real and simulated datasets in this section, we have used the functions \texttt{mincontrast} and \texttt{linearpcf} from \texttt{spatstat} for calculating the minimum contrast estimates and the non-parametric estimate of the pcf (the latter is only available in the case of the geodesic metric), respectively, as well as a range of other minor functions from \texttt{spatstat} for handling point processes on linear networks. Furthermore, we have used the function \texttt{global\_envelope\_test} from the \texttt{GET} package for global envelope tests. The rest of the code (such as the estimation of the pcf in the case of the resistance metric, or the estimation of the $F$, $G$, and $J$-functions) we have implemented ourselves, and it is available in our package \texttt{coxln}. 

\subsection{Analysis of Chicago crime dataset}

\cite{AngEtAl} used the empirical $K$-function together with simulations to show that the Chicago crime dataset (shown in the left panel of Figure~\ref{fig:data} and available in \texttt{spatstat}) was more clustered than a homogeneous Poisson process while an inhomogeneous Poisson process with log-quadratic intensity fitted by maximum likelihood provided a better fit. They remarked that the latter model was only provisional and  shown for demonstration. 

For the analysis of the street crimes in this paper we fitted instead all three model classes (LGCP, ICP, and PCPP) given in Section~\ref{s:cox}. 
For the Gaussian process underlying the Cox process models, we used for simplicity a constant mean and an isotropic exponential covariance function with metric $d=d_{\mathcal R}$ 
(since the network is not a 1-sum of trees and cycles, the geodesic metric will not give well-defined models). Thus we have (up to) four unknown parameters in each model: $\rho$, the intensity of the point process; $\sigma^2$, the variance of the Gaussian process (this is not  a parameter in the PCPP model); $s$, the scaling of the exponential covariance function; and $h$, the number of Gaussian processes used in the model (this is not  a parameter in the LGCP model). 

To estimate the parameters of all three models, we used the unbiased estimate $\hat{\rho} ={n(x)}/{|L|}= 0.00372$ for the intensity and estimated the remaining parameters by the 
minimum contrast method based on the pair correlation function, cf.\ \eqref{e:contrast}. We estimated $g_0$ non-parametrically on the interval $r\in[0,100]$, but since the shape of the non-parametric estimate mostly resembled the shapes of the theoretical pair correlation functions for the three models on the interval $r\in[20,100]$, in \eqref{e:contrast} we 
let $a_1=20$ and $a_2=100$ while $p$ and $q$ were given by the default values in \texttt{spatstat}. 

The estimated parameter values are shown in Table~\ref{tab:chicago-estimates}. Figure~\ref{fig:chicago-pcf} shows the non-parametric estimate together with the estimated pair correlation functions for each of the three models, where 
the estimated pair correlation functions are very similar for the LGCP and the ICP, while the pair correlation function for the PCPP deviates from the other two at low distances. 


\begin{table}
	\centering
	\begin{tabular}{l||c|c|c|c}
		Model & $\hat\rho$ & $\hat{\sigma}^2$ & $\hat s$ & $\hat h$ \\
		\hline\hline
		LGCP  & 0.00372 & 1.70 & 0.0213  & -- \\
		ICP   & 0.00372 & 22.8 & 0.00747 & 2 \\
		PCPP   & 0.00372 &   --  & 0.00988 & 1
	\end{tabular}
	\caption{Parameter estimates for the LGCP, ICP, and PCPP models for the Chicago crime dataset.  A dash  indicates that the parameter is not present in the model.} 
	\label{tab:chicago-estimates}
\end{table}

\begin{figure}
	\begin{center}\includegraphics[width=7cm]{"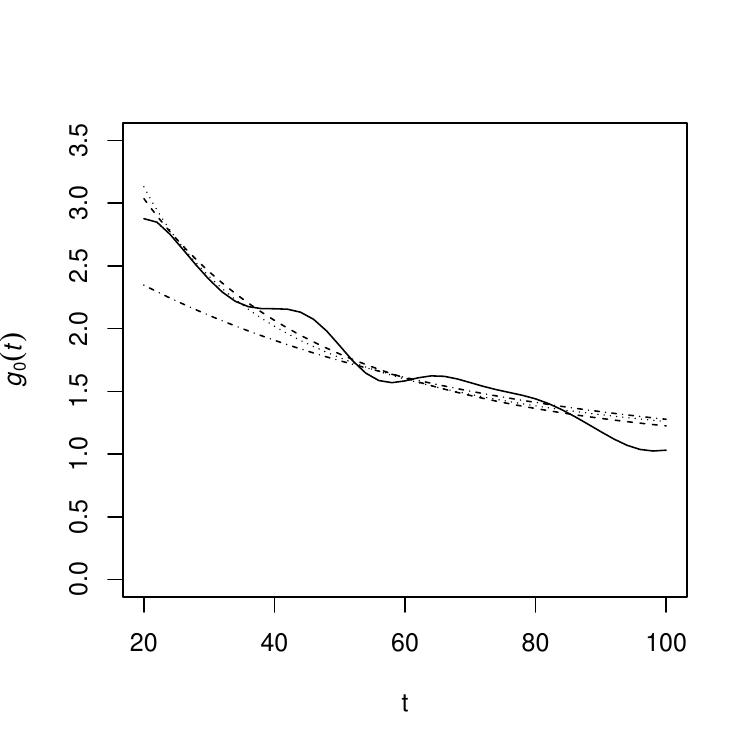"}\end{center}
	\caption{Pair correlation functions for the Chicago crime dataset: Non-parametric estimate (solid curve) and curves estimated by minimum contrast for a LGCP (dashed curve), an ICP (dotted curve), and a PCPP (dashed-dotted curve) model.}
	\label{fig:chicago-pcf}
\end{figure}

The left column of Figure~\ref{fig:chicago-sim} shows a simulation from each of the three fitted models. None of the simulations show strong deviation from the data, although the ICP does show a tendency to have small densely packed clusters of points that are not present in the data. Moreover, the right column of Figure~\ref{fig:chicago-sim} shows the 95\% global envelopes based on a concatenation of the empirical $F,G,J$-functions (based on 999 simulations) discussed at the end of Section~\ref{s:est-check}. The LGCP provides the best fit with a $p$-value of 0.155 for the global envelope test. The ICP provides a rather bad fit to the data with a $p$-value of just 0.004, and the envelope shows a clear discrepancy between data and the model at distances below 50, where the discrepancy indicates more clustering in the fitted ICP model than in the data.
The PCPP model shows a decent fit with a $p$-value just below 0.05 due to the $F$-function going outside the envelope at distance 100. 

\begin{figure}
	\begin{center}
		\includegraphics[width=6.5cm]{"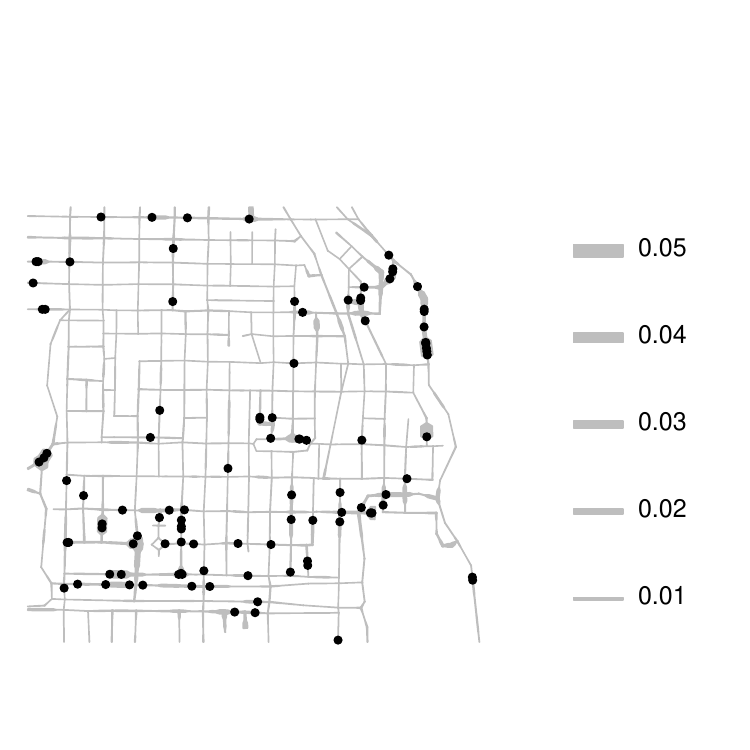"} 
		\includegraphics[width=6cm]{"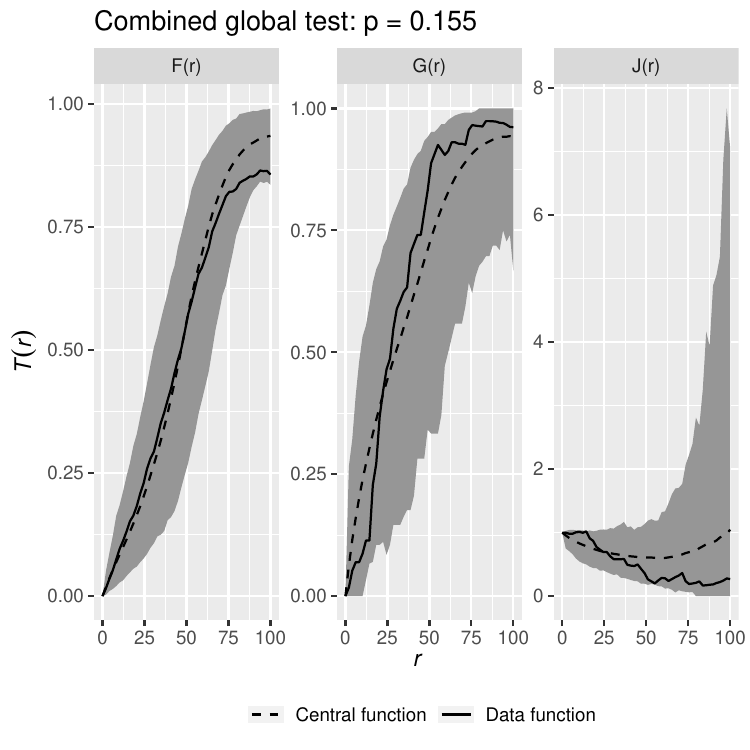"}
		\includegraphics[width=6.5cm]{"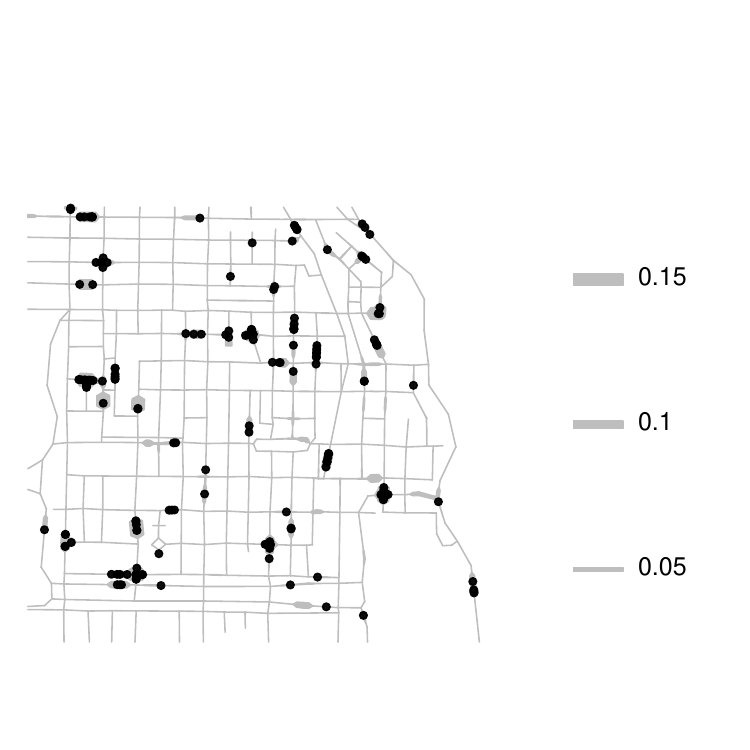"}
		\includegraphics[width=6cm]{"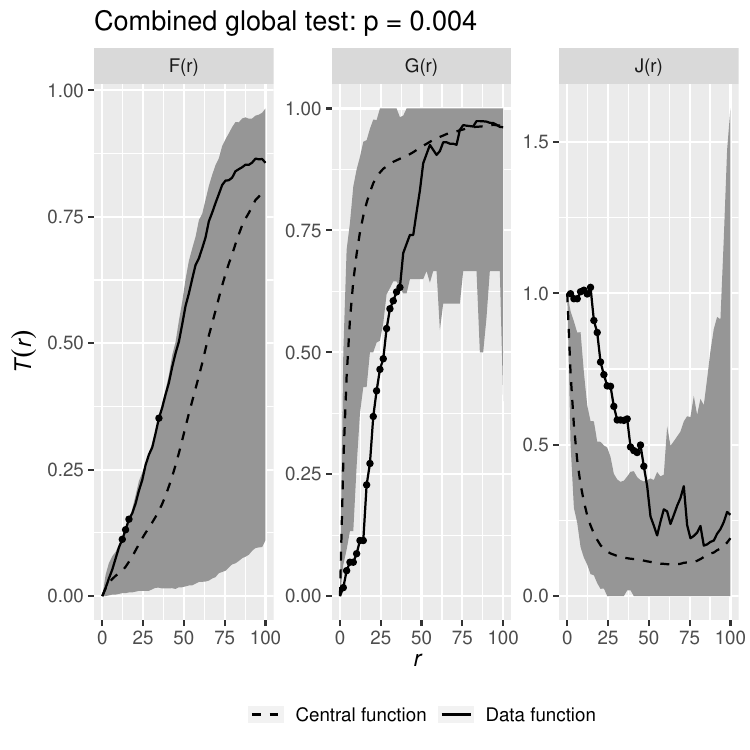"}
		\includegraphics[width=6.5cm]{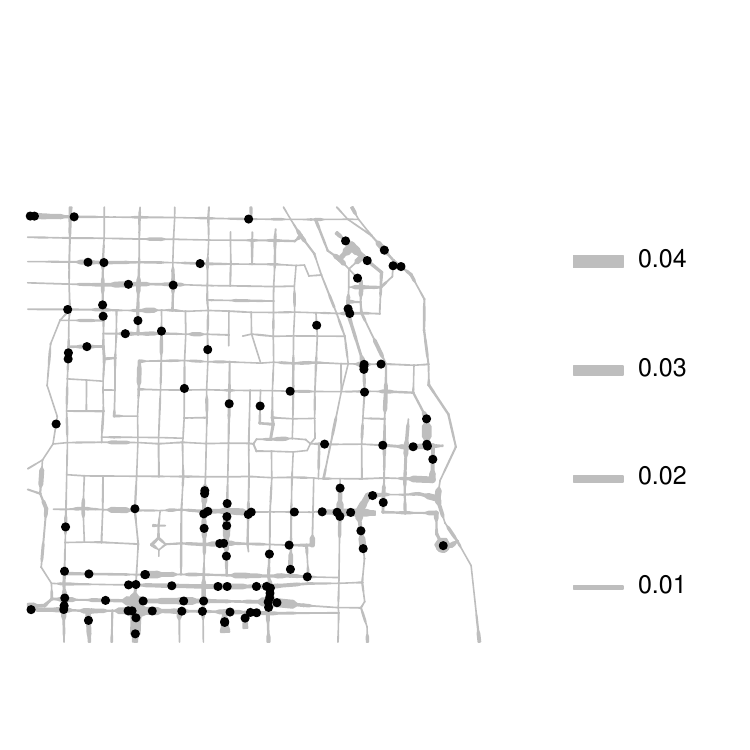}
		\includegraphics[width=6cm]{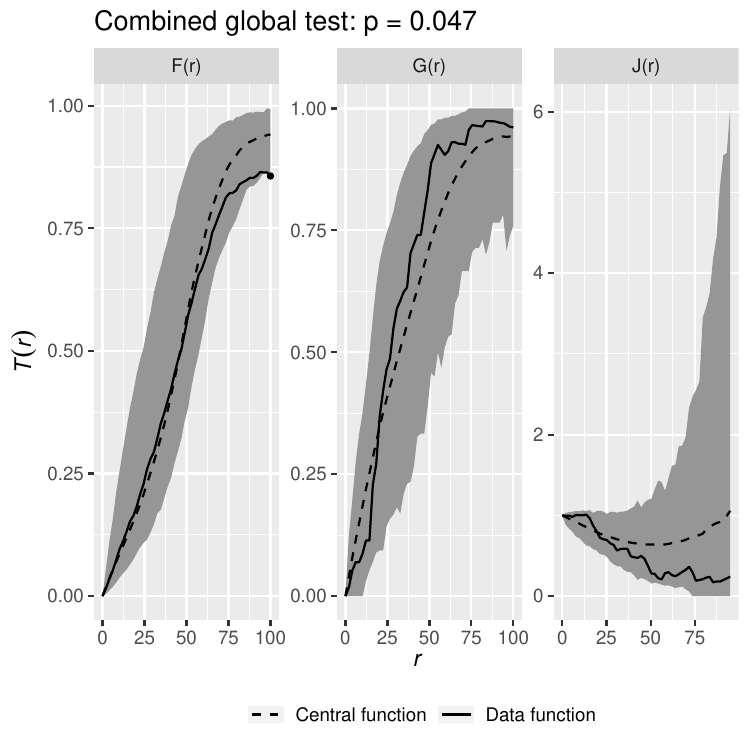}
	\end{center}
	\caption{Left column: Simulations of point patterns under the estimated LGCP (upper row), ICP (middle row) and PCPP (lower row) models for the Chicago crime dataset. The width of the lines shows the simulation of the underlying transformed Gaussian processes (the random intensity functions), while the points show the simulated point pattern. Right column: 95\% global envelopes for the fitted LGCP (upper row), ICP (middle row), and PCPP (lower row) models for the Chicago crime dataset. The $p$-values of the global envelope test are shown above each plot.}
	\label{fig:chicago-sim}
\end{figure}
 


\subsection{Dendrite data}\label{s:den-an}

\citet{HeidiMe} analysed several point pattern datasets given by spine locations on different dendrite trees which were identified by linear networks. For each dataset they fitted an inhomogeneous ICP with $h=1$, $c$ an isotropic exponential covariance function with metric $d=d_{\mathcal G}=d_{\mathcal R}$, and an inhomogeneous intensity function given by a constant intensity $\rho_{\mathrm{mb}}>0$ on the main branch and a different constant intensity $\rho_{\mathrm{sb}}>0$ on the side branches. 

In this section we restrict attention to the dataset for dendrite number five in \citet{HeidiMe} (the right panel of Figure~\ref{fig:data}) modelled by 
an ICP with $h=1$ and considering different covariance function models. Hence we
estimated the intensity parameters by $\hat\rho_{\mathrm{mb}}= {n(x_{\mathrm{mb}})}/{|L_{\mathrm{mb}}|}=0.119$ and $\hat\rho_{\mathrm{sb}} = {n(x_{\mathrm{sb}})}/{|L_{\mathrm{sb}}|}=0.184$ where $L_{\mathrm{mb}}$ denotes the main branch, $L_{\mathrm{sb}}$ the union of side branches, and $x_{\mathrm{mb}}$ and $x_{\mathrm{sb}}$ the point patterns restricted to these two subsets of $L$. 



First, we let $c(u,v)=\sigma^2\exp(-sd(u,v))$ be an isotropic exponential covariance function with $d=d_{\mathcal G}=d_{\mathcal R}$. Then we estimated $(\sigma^2,s)$ following the recommendations in \citet{HeidiMe}, i.e., we 
used the minimum contrast method with  $a_1=0$, $a_2=50$, $p=2$, and $q=1$ in \eqref{e:contrast}.
The obtained estimate $(\hat\sigma^2,\hat s)$ is given in Table~\ref{tab:neu-estimates} and the corresponding estimated pcf is shown in Figure~\ref{fig:neu-pcf} (the dotted curve). The estimate 
deviates a bit from that obtained in \citet{HeidiMe}, which is due to a different implementation of the non-parametric estimation of the pair correlation function.
Figure~\ref{fig:neu-FGJ} (left panel) shows the 95\% global envelope based on a concatenation of the empirical $F,G,J$-functions (based on 999 simulations). This shows a satisfactory fit with a $p$-value of $0.26$ for the global envelope test. The $p$-value deviates substantially from the value obtained in \citet{HeidiMe}, which may be due to the different estimates of $(\sigma^2,s)$ and a different number of simulations used in the global envelope test.

\begin{table}
	\centering
	\begin{tabular}{l||c|c}
		Covariance function &  $\hat{\sigma}^2$ & Other estimates\\
		\hline\hline
		Exponential   & $3.90$ & $\hat{s}=0.0356$ \\
		Gamma Bernstein CDF  & $3.91$ &  $(\hat\tau,\hat\phi)=(22.6, 626)$\\ 
		Inverse gamma Bernstein CDF & $3.90$ &  $(\hat\tau,\hat\phi)=(163, 5.79)$ \\
		Generalised inverse Gaussian Bernstein CDF & $3.90$ & $(\hat\psi,\hat\chi,\hat\lambda) = (213, 35.8, -499)$ \\
		Inverse gamma Bernstein CDF with fixed $\tau$  & $4.63$ & $(\hat\tau,\hat\phi)=(2, 0.0188)$ \\
	\end{tabular}
	\caption{Parameter estimates for the ICP models with various covariance functions for the dendrite dataset.} 
	\label{tab:neu-estimates}
\end{table}

\begin{figure}
	\begin{center}
		\includegraphics[width=6cm]{"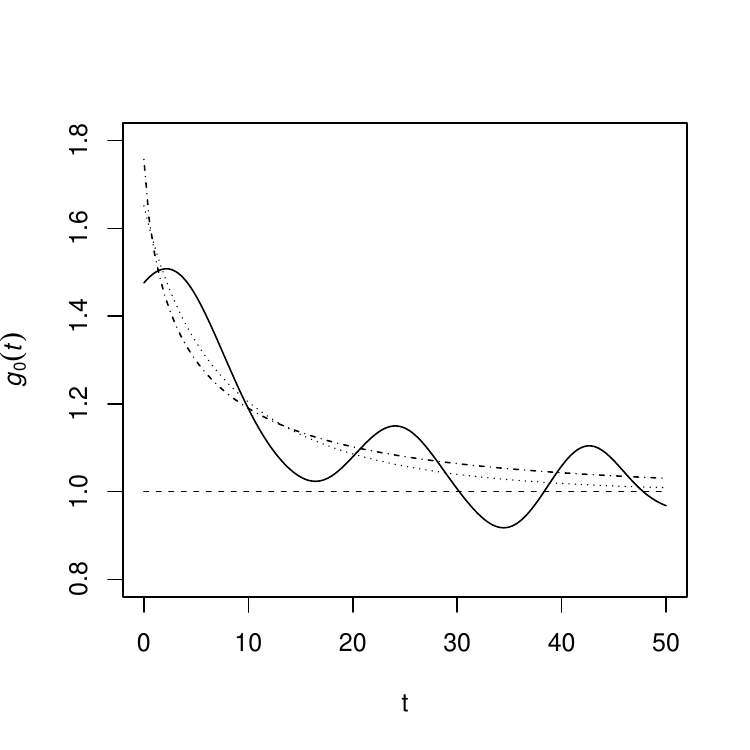"}\end{center}
	\caption{Pair correlation functions for the dendrite dataset: Non-parametric estimate (solid curve), and curves estimated by minimum contrast for ICP with exponential covariance  (dotted curve) and covariance function with inverse gamma Bernstein density with fixed $\tau=2$  (dotted-dashed curve).}
	\label{fig:neu-pcf}
\end{figure}

\begin{figure}
	\begin{center}
	\includegraphics[width=6cm]{"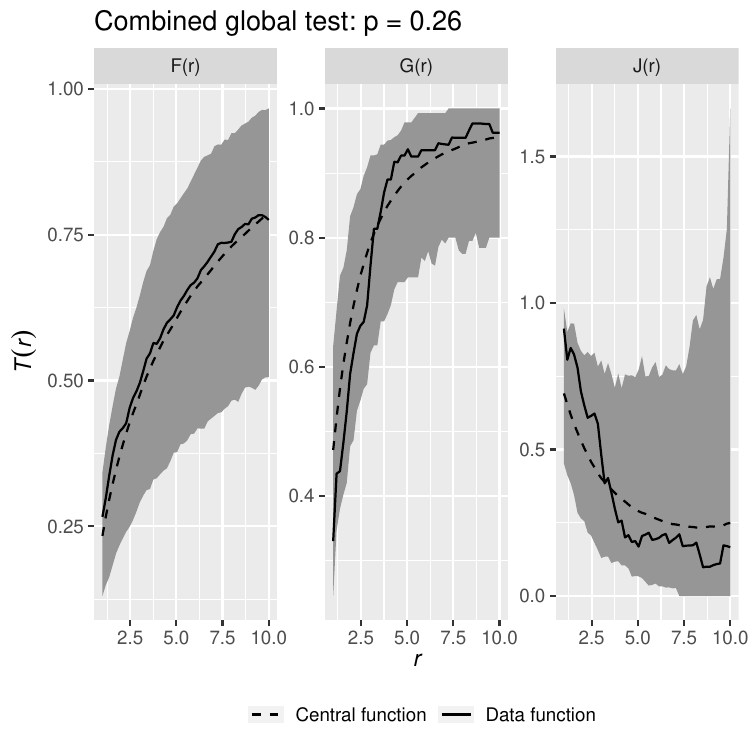"}
	\includegraphics[width=6cm]{"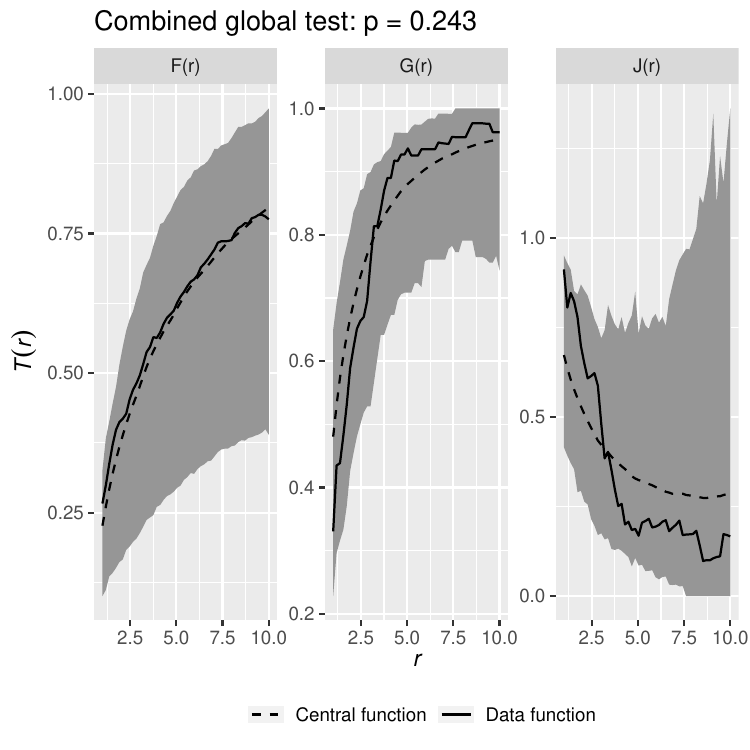"}
	\end{center}
	\caption{95\% global envelopes for the fitted ICP models with an exponential covariance function (left) and a covariance function with an inverse gamma Bernstein CDF with $\tau=2$ (right) for the dendrite dataset. The $p$-values of the global envelope test are shown above each plot.}
	\label{fig:neu-FGJ}
\end{figure}

Second, we let $c$ be one of the three covariance functions given in Example~\ref{ex:0}, i.e., those with a gamma, an inverse gamma, and a generalised inverse Gaussian Bernstein CDF. We used the same minimum contrast procedure as above for estimating the parameters, thereby obtaining the estimates in Table~\ref{tab:neu-estimates}. 
For all three cases of the estimated Bernstein CDFs, the variance is close to zero meaning that the distributions are almost degenerate and the corresponding covariance functions are very close to exponential covariance functions. Since the mean values of these distributions are very close to $\hat s$ from the estimated exponential covariance function, all three estimated covariance functions are very close to the estimated exponential covariance function, thus leading to the same estimated model. 
This suggests possible problems with the estimation procedure, which we will explore in a simulation study in Section~\ref{s:simstudy}, or with unidentifiability of the parameters of the Bernstein CDF.
 
To try out a model which do not become almost identical to the model using the exponential covariance function, we considered a covariance function with an inverse gamma Bernstein CDF with $\tau=2$ fixed (corresponding to an inverse gamma distribution with an infinite variance). 
Table~\ref{tab:neu-estimates} shows the minimum contrast estimate and Figure~\ref{fig:neu-pcf} shows 
 the corresponding estimated covariance function. This covariance function has a larger variance parameter $\sigma^2$ and a heavier tail than the estimated exponential covariance function. The 95\% global envelope in Figure~\ref{fig:neu-FGJ} (right panel) and the $p$-value of $0.243$ for the global envelope test reveal that the covariance function with inverse gamma density and fixed parameter $\tau=2$ provides a similar good fit as the exponential covariance function. This observation is further discussed in Section~\ref{s:7.3}.

\subsection{Simulation study}\label{s:simstudy}

The analysis of the dendrite dataset showed the problem that we in practice get an exponential covariance function when we fit the covariance functions given in Example~\ref{ex:0}. It should be noted that all these covariance functions have the exponential covariance function as a limiting case, which makes it possible to get arbitrarily close to an exponential covariance function in the estimation procedure, while this was not the case when the $\tau$ parameter was fixed in the analysis of the dendrite data.

To explore whether this was a general problem or simply applied to the dendrite dataset, we made a number of simulations using covariance functions which were different than the exponential case to see if we still obtained exponential covariance functions from the estimation procedure. Specifically we took the network used in the dendrite dataset and simulated 1000 simulations of an ICP with a homogeneous intensity function and a covariance function with inverse gamma Bernstein CDF. We did this for various combinations of parameters, where $\sigma^2\in\{10^{-1},10^0,10^1,10^2,10^3,10^4\}$ and $\tau\in\{1.1,1.5,2,5\}$ while the rest of the parameters were given by $(\rho, h, \phi) = (1, 1, 0.02)$. The mean number of points  is $\mathbb E N(L) = \rho|L| = 736$ for all the chosen parameter settings. For each simulation we fitted two models using minimum constrast estimation (using the same values of $a_1,a_2,q,p$ as in Section~\ref{s:den-an}): an ICP with an inverse gamma Bernstein CDF with $h$ and $\tau$ fixed at their true values, and an ICP with an exponential covariance function with $h=1$ fixed. 
For each simulation, we calculated the non-parametric estimate of the pcf and used $D$ in \eqref{e:contrast} for calculating its distance to each of the two pcfs for the estimated models; using an obvious notation, these distances are denoted $D_\text{Exp}$ and $D_\text{IG}$. For each combination of parameters $\sigma^2$ and $\tau$, Table~\ref{tab:comparison} shows the percentage of simulations where $D_\text{IG}>D_\text{Exp}$, i.e., the cases where the non-parametric estimate of the pcf resembles an exponential covariance function more than a covariance function with inverse gamma Bernstein CDF with $\tau$ equal to the value used in the simulations. To quantify how far the ICP is from a Poisson process, the values of $\varphi$ given by \eqref{e:varphiICP} and $p_\textrm{ms}$ given by \eqref{e:pms} are also shown in Table~\ref{tab:comparison}. The row in the table corresponding to $\sigma^2=10^{-1}$ contains values close to $50\%$, which makes sense since $g_0(0)-1=0.00416$ or $p_\mathrm{ms}=0.912$ reveal that the simulated processes are almost Poisson.
In the rest of the table the non-parametric estimate are typically closest to an exponentiel covariance function, showing that there indeed is a tendency for getting an estimated model corresponding to an exponential covariance function (although the values seem to decrease in the last row). 

\begin{table}
	\centering
	\begin{tabular}{c|c|c|c|c|c|c}
		$\sigma^2$ & $\varphi$ & $p_\textrm{ms}$ & $\tau=1.1$ & $\tau=1.5$ & $\tau=2$ & $\tau=5$  \\
		\hline\hline
		$10^{-1}$  & 0.00416 & 0.912 & 56.6\% & 54.3\% & 53.3\% & 48.3\% \\
		$10^0$     & 0.155   & 0.577 & 81.1\% & 84.6\% & 86.8\% & 81.6\% \\
		$10^1$     & 1.40    & 0.218 & 74.3\% & 80.1\% & 85.8\% & 89.2\% \\
		$10^2$     & 6.12    & 0.0705& 76.1\% & 82.2\% & 82.6\% & 87.1\% \\
		$10^3$     & 21.4    & 0.0223& 82.5\% & 79.8\% & 79.0\% & 81.0\% \\
		$10^4$     & 69.7    &0.00707& 55.6\% & 60.1\% & 66.1\% & 77.3\% \\
	\end{tabular}
	\caption{The percentage of simulations of an ICP with $D_\text{IG}>D_\text{Exp}$ for various combinations of $\sigma^2$ and $\tau$ values, where the corresponding values of $\varphi$ and $p_\mathrm{ms}$ are shown.} 
	\label{tab:comparison}
\end{table}

We also made similar simulation studies for LGCPs and PCPPs. These showed similar values to those in Table~\ref{tab:comparison} for the LGCP, while the values where significantly smaller for the PCPP (typically in the range $30\%$-$40\%$ when $\varphi=2$ is as large as possible in this model).

The results of this section are further discussed in Section~\ref{s:7.3}.



\section{Concluding remarks}\label{s:conclusion}

Our results and considerations on point processes on linear networks in the previous sections give rise to various conclusions and open research problems which we discuss in Sections~\ref{s:7.1}--\ref{s:7.4} below.

\subsection{New point process models}\label{s:7.1}

\citet{Stationary} mentioned the lack of repulsive models on linear networks. An interesting case is a determinantal point process (DPP) which to the best of our knowledge has yet not been investigated in connection to linear networks. 
 For any given covariance function $c$ on $L$, a well-defined DPP will be specified by the density 
\begin{equation}\label{e:dppdens}
f(\{u_1,...,u_n\})\propto \mathrm{det}\left(\{c(u_i,u_j)\}_{i,j=1,...,n}\right)
\end{equation}
with respect to the unit rate Poisson process on $L$. A DPP is a model for repulsiveness (e.g.\ $g\le1$) and it {possesses} a number of appealing properties.
See
\citet{FredEtAl} (and the references therein) who considered DPPs with points in $\mathbb R^k$ but many things are easily modified to DPPs on linear networks.

Several facts are interesting: If $c$ is isotropic then the pcf of the DPP is isotropic, and e.g.\ 
$c$ could be specified by one of 
the isotropic covariance functions given in Table~\ref{tab:covariances-sphere} and Example~\ref{ex:0}. 
The normalizing constant which is omitted in \eqref{e:dppdens} can be expressed in terms of the eigenvalues of a spectral decomposition of $c$. This spectral decomposition is also needed when specifying the $n$-th order intensity functions and an efficient simulation algorithm of the DPP.  However, it is an open problem how to determine the eigenvalues and eigenfunctions of the spectral representation, in particular if $c$ is required to be isotropic. Until this problem has been solved, we need to work with the unnormalized density (in the right hand side of \eqref{e:dppdens}) and to use MCMC methods for approximating the normalizing constant \citep{Geyer,MW04}.

The recent paper \citet{BSW} on a new class of Whittle–Mat{\'e}rn GPs on compact metric graphs is interesting for several reasons including the following. 
The model class is flexible and contains differentiable GPs; the precision matrix at the vertices can be quickly calculated; and
Markov properties of the GP can be used for computationally efficient inference. Thus the Whittle–Mat{\'e}rn GPs may 
 serve as an interesting alternative to the GPs considered in the present paper. Although 
the
Whittle–Mat{\'e}rn covariance function is not isotropic, and hence in connection to items (I)--(VI) in Section~\ref{s:intro} may appear to be less useful, 
this could very well be an advantage as pointed out after item (VI).
 

\subsection{The choice of metric for isotropic covariance and pair correlation functions}\label{s:7.2}

Isotropic covariance functions are available for larger classes of linear networks with respect to the resistance metric than the geodesic metric, cf.\ Theorem~\ref{t:1}. Therefore, we are able to obtain LGCPs, ICPs, and PCPPs for larger classes of linear networks with respect to the resistance metric than the geodesic metric. 
Indeed, this kind of flexibility assumes that one uses the same covariance but in different metrics, which is a bit specific. Could we obtain the same flexibility with a different covariance function in the geodesic metric? How do we prove that it will actually be a valid covariance function?

Various other comments in the previous sections highlighted the interest of $d_{\mathcal R}$ compared to $d_{\mathcal G}$. In addition we notice the following.

Suppose $L$ is not a tree and that both $c_0(d_{\mathcal G}(u,v))$ and $c_0(d_{\mathcal R}(u,v))$ are well-defined isotropic covariance functions.
(Recall that $d_{\mathcal G}\ge d_{\mathcal R}$ with equality if and only if $L$ is a tree, cf.\ (C) in Theorem~\ref{t:res-m}.) 
Suppose also that $c_0> 1$ and $c_0$ is a decreasing function, so $c_0(d_{\mathcal G}(u,v))\le c_0(d_{\mathcal R}(u,v))$. This is 
typical for the covariance function of the GP underlying a LGCP, ICP, or PCPP
(an exception is a LGCP if $c_0$ can be negative but usually {$c_0\ge0$}). Consider the two LGCPs, ICPs, or PCPPs obtained by using the same type of transformed isotropic GP (or GPs) with the same mean and $c_0$-function but using the different metrics $d=d_{\mathcal G}$ and $d=d_{\mathcal R}$. Using an obvious notation, let us denote 
these two Cox processes by $X_{\mathcal G}$ and $X_{\mathcal R}$ and their corresponding pair correlation functions by 
$g_{\mathcal G}(u,v)=g_0(d_{\mathcal G}(u,v))$ and $g_{\mathcal R}(u,v)=g_0(d_{\mathcal R}(u,v))$. It follows from \eqref{e:pairLGCP}, \eqref{e:g0ICP}, and \eqref{e:pairPCPP} that   
 {$g_{\mathcal G}(u,v)\le g_{\mathcal R}(u,v)$}, indicating that $X_{\mathcal R}$ is able to model a higher degree of clustering than  
 $X_{\mathcal G}$.
 
%
%

On the other hand, for a DPP, we have always that $g_0\le 1$, and typically $g_0< 1$ and $g_0$ is an increasing function. So for two such DPPs with the same $g_0$-function but given by the different metrics $d_{\mathcal G}$ and $d_{\mathcal R}$, respectively, the DPP using $d_{\mathcal R}$ is able to a higher degree of repulsiveness than the DPP using $d_{\mathcal G}$. 


\subsection{Estimation}\label{s:7.3}


For parameter estimation in this paper we used minimum contrast estimation based on the pcf which worked well when we used an exponential covariance function, but this approach was not able to identify the parameters in the covariance functions given by the Bernstein CDFs in Example~\ref{ex:0} unless we fixed a subparameter, cf.\ Sections~\ref{s:den-an} and \ref{s:simstudy}. We are not sure if this is a problem of unidentifiability or caused by the choice of estimation procedure for the following reasons. 

On the one hand,
\cite{HeidiMe} concluded that for the dendrite data minimum contrast estimation based on the pcf performs better than if it is based on the $K$-function, and the alternative method based on  composite likelihood estimation is less reliable than minimum contrast estimation for the exponential covariance function, thus suggesting that this approach will not improve estimation for more complex covariance functions either.
This is also in line with the comments in \citet{BadEtAl2022} on estimation procedures. 
In particular, \citet[][Section 12]{BadEtAl2022} concluded that Neyman-Scott point process models are poorly identified under weak clustering, irrespective of the model parametrization, and this is not due to faults in currently available fitting
methods (they considered the Euclidean state space case but we do not expect the choice of space to be of importance for this conclusion).

On the other hand, if the problem is unidentifiability, one idea for estimating the parameters of the Bernstein CDFs could be to  
to use shrinkage estimators involving a penalty on cluster scale as was done in \citet{BadEtAl2022} using the parameter $\varphi$ in \eqref{e:varphi}.
However, we also remade the simulation studies behind Table~\ref{tab:comparison} in the case where we fixed the parameter $\sigma^2$ at its true value, which did not improve the estimation of the parameters in the Bernstein CDF, thus suggesting this approach might not work.

A more careful study of whether the problem is the choice of estimation procedure or unidentifiability, extending 
the methods of \citet{BadEtAl2022} to our setup, is left for future research. 
More ambitiously than minimum contrast, composite likelihood, and other moment based estimation procedures \citep{MW17}, future work may also investigate likelihood-based inference for parametric point process models on linear networks. This could include the adaptation of missing data MCMC methods for maximum likelihood \citep{Geyer,MW04} and Bayesian approaches \citep{RueEtAl,GT} to linear networks. 

\subsection{Miscellaneous}\label{s:7.4}

In \eqref{e:bound-prob} we provide a bound in probability of the approximation error of Algorithms 1 and 2 due to the discretization. This may be extended to a bound for the maximal approximation error when considering variables $(Y(u_1),...,Y(u_n))$ with $u_1,...,u_n$ on the same line segment of $L$ (using a similar approach as in Equation~(4.3) in Wood and Chan~(1992) which is based on an inequality for multivariate normal probabilities on rectangles, see Chapter~2 in Tong~(1982)).\nocite{WoodChan} \nocite{tong} Furthermore, for Algorithm 3, it could be interesting to establish convergence rates. 
  
The special results in Section~\ref{s:LGCP} for LGCPs on linear networks (as well as LGCPs defined on other state spaces like $\mathbb R^k$ and $\mathbb S^k$) could possibly be exploited for developing techniques of model fitting or checking, where we have the following results in mind. 
Assume the covariance function of the underlying GP is isotropic. The $K$-functions based on $X$ and $X_u$ agree, so if $X$ is observed within a bounded window $W\subseteq L$, empirical estimators $\hat K$ and $\hat K_u$ for $u\in X_W$ are expected to be close. 
Moreover, the special structure in 
\eqref{e:n-LGCP} implies that 
\[\frac{\rho(u_1,u_2,u_3)}{\rho(u_1)\rho(u_2)\rho(u_3)}=g_0(d(u_1,u_2))g_0(d(u_1,u_3))g_0(d(u_2,u_3)).\]
This structure was exploited in \citet[][]{LGCP} for LGCP with points in $\mathbb R^k$ to construct a third-order moment characteristic which is useful for model checking, but it remains to consider the case of a LGCP on a linear network.

We have given various references for non-parametric estimation of the intensity, pair correlation, and $K$-functions on linear networks, cf.\ Sections~\ref{s:3.2} and \ref{s:3.2.3}. In the inhomogeneous case these estimators depend `locally' on the intensity function at the individual
observed points. Recently, when considering point processes on Euclidean spaces, \citet{ShawEtAl} demonstrated the advantages of introducing new global estimators over the
existing local estimators. It remains to consider the case of point processes on other spaces including linear networks.

\bibliography{references}

\end{document}